\documentclass[reqno,11pt]{amsart}

\usepackage[T1]{fontenc}
\usepackage[utf8]{inputenc}
\usepackage{lmodern}
\usepackage[a4paper,margin=1in]{geometry}
\usepackage{amsmath,amssymb,amsthm,mathtools,mathrsfs}
\usepackage{microtype}
\usepackage{enumitem}
\usepackage{cite}
\usepackage{hyperref}
\hypersetup{
 colorlinks=true,
 linkcolor=blue,
 citecolor=blue,
 urlcolor=blue,
 pdftitle={Stable and Unstable Potential-Well Dynamics for an Indirectly Damped Wave--MGT System with General Focusing Sources},
 pdfauthor={Tae Gab Ha},
 pdfsubject={Potential-well dynamics, indirect stabilization, and finite-time blow-up for a wave--MGT system with general focusing sources},
 pdfkeywords={wave--MGT system, indirect damping, potential well, high-frequency flux, finite-time blow-up}
}
\allowdisplaybreaks

\newtheorem{theorem}{Theorem}[section]
\newtheorem{proposition}[theorem]{Proposition}
\newtheorem{lemma}[theorem]{Lemma}
\newtheorem{corollary}[theorem]{Corollary}
\theoremstyle{definition}
\newtheorem{definition}[theorem]{Definition}
\theoremstyle{remark}
\newtheorem{remark}[theorem]{Remark}

\numberwithin{equation}{section}

\newcommand{\R}{\mathbb R}
\newcommand{\N}{\mathbb N}
\newcommand{\Hcal}{\mathcal H}
\newcommand{\Wcal}{\mathcal W}
\newcommand{\Ucal}{\mathcal U}
\newcommand{\Ncal}{\mathcal N}
\newcommand{\Xcal}{\mathcal X}
\newcommand{\Ecal}{\mathcal E}
\newcommand{\Escr}{\mathscr E}
\newcommand{\Pcal}{\mathcal P}
\newcommand{\Rcal}{\mathcal R}
\newcommand{\Acal}{\mathcal A}
\newcommand{\Lcal}{\mathcal L}
\newcommand{\Vcal}{\mathcal V_f}
\newcommand{\norm}[1]{\left\lVert #1\right\rVert}
\newcommand{\abs}[1]{\left\lvert #1\right\rvert}
\newcommand{\ip}[2]{\left(#1,#2\right)}
\newcommand{\dual}[2]{\left\langle #1,#2\right\rangle}
\newcommand{\dd}{\,\mathrm d}

\title[Potential-well dynamics for a wave--MGT system]
{Stable and Unstable Potential-Well Dynamics for an Indirectly Damped Wave--MGT System with General Focusing Sources}

\author{Tae Gab Ha}
\address{Department of Mathematics and Institute of Pure and Applied Mathematics, Jeonbuk National University, Jeonju 54896, Republic of Korea}
\email{tgha@jbnu.ac.kr}


\begin{document}

\begin{abstract}
We study a conservative semilinear wave equation coupled through a zero-order interaction to a dissipative Moore--Gibson--Thompson equation on a bounded domain. The wave component carries no direct damping and is driven by a general focusing source $f(u)$. The augmented variable $w=v+\tau v_t$ reveals an exact coupled energy and a coercive potential-well geometry. The source assumptions are formulated through
\[
 H_\theta(s)=\frac1\theta sf(s)-F(s),
 \qquad F(s)=\int_0^s f(r)\,\dd r,
 \qquad \theta>2.
\]
Under $L^2$-subcritical $C^1$ growth, smallness at the origin, nonnegativity and radial monotonicity of $H_\theta$, and a nontrivial focusing condition, we establish local well-posedness, the exact energy identity, and a continuation alternative for arbitrary finite-energy data. For nonzero coupling, the linearized semigroup is strongly stable, whereas a wave-branch expansion precludes uniform exponential stability and positive-time compactness. Below the coupled well depth, the stable set is positively invariant and generates global solutions, while data with negative Nehari functional blow up in finite time without a sign condition on the initial velocities. At the critical level $E(0)=d$, nonzero coupling yields a complete trichotomy into stable entry, finite-time blow-up, or a stationary Nehari ground state. Under the same coupling condition, every stable trajectory converges weakly to zero without any compactness hypothesis. A renormalized high-frequency identity shows that vanishing of the accumulated nonlinear high--low flux is equivalent to relative compactness of the orbit and to strong convergence in the natural energy space. We give several sufficient criteria, including finite total variation of the nonlinear force in $L^2(\Omega)$. Regular stable data generate global regular solutions; under the same source-variation condition, a corrected high-order energy yields a uniform bound in the regular phase space and $\Delta v_t\in L^2(0,\infty;L^2(\Omega))$. Independently of the Nehari level, a shifted MGT-history argument yields a data-only positive-energy blow-up criterion and blow-up data on every prescribed nonnegative energy shell. Pure-power, logarithmic, and positive finite sums of powers are covered.
\end{abstract}
\maketitle

\section{Introduction}
Let $\Omega\subset\R^n$, $n\ge1$, be a bounded domain with $C^2$ boundary. We study
\begin{equation}\label{eq:original-system}
\begin{cases}
 u_{tt}-\Delta u+\alpha(v+\tau v_t)=f(u),& \text{in }\Omega\times(0,\infty),\\
 \tau v_{ttt}+v_{tt}-\Delta v-b\Delta v_t+\alpha u=0,& \text{in }\Omega\times(0,\infty),\\
 u=v=0,& \text{on }\partial\Omega\times(0,\infty),\\
 (u,u_t)(0)=(u_0,u_1),\\
 (v,v_t,v_{tt})(0)=(v_0,v_1,v_2).
\end{cases}
\end{equation}
Here $\tau>0$ is the relaxation parameter, $b>0$ is the MGT coefficient, and $\alpha\in\R$ is the coupling strength. Throughout the paper,
\begin{equation}\label{eq:parameter-assumptions}
 b>\tau>0,\qquad \abs{\alpha}<\lambda_1,
\end{equation}
where $\lambda_1$ is the first Dirichlet eigenvalue of $-\Delta$ on $\Omega$.

The second equation in \eqref{eq:original-system} is the normalized Moore--Gibson--Thompson equation. It arises in nonlinear acoustics and high-intensity ultrasound when finite thermal relaxation is retained, and its third-order time structure removes the instantaneous-relaxation idealization inherent in classical Fourier-based acoustic models. The linear and semilinear Cauchy theories, boundary-value problems, memory effects, and nonlinear global dynamics of MGT and Jordan--MGT equations have been developed extensively; see, among others, \cite{KaltenbacherLasieckaMarchand,MarchandMcDevittTriggiani,DellOroLasieckaPata,DellOroLasieckaPata2020,LasieckaWang,BucciEller2021,ChenIkehata2021,RackeSaidHouari2021,SaidHouari2022,ChenPalmieri}.

MGT-type constitutive laws also occur outside their original acoustic setting. Recent formulations of MGT thermoviscoelasticity and mixtures of an MGT viscoelastic constituent with an elastic solid show that a third-order relaxation field can be coupled consistently to a conservative elastic component; see \cite{ContiPataPellicerQuintanilla2021,FernandezQuintanilla2022}. Motivated by this continuum-mechanical interpretation, we regard \eqref{eq:original-system} as a minimal energy-consistent reduced model for the interaction of a conservative elastic wave field with an MGT-type relaxing internal field through a distributed reciprocal coupling. In this reading, $u$ represents the conservative wave displacement or amplitude, $v$ is the relaxing internal variable, and $w=v+\tau v_t$ is the associated relaxed, or recoverable, response. The interaction energy $\alpha\ip{u}{w}$ generates the reciprocal zero-order forces $\alpha w$ and $\alpha u$. Accordingly, \eqref{eq:original-system} is used here as a canonical energy-consistent prototype for the indirect transfer of dissipation between a conservative wave field and a relaxing internal field.

The parameter restrictions in \eqref{eq:parameter-assumptions} have corresponding structural meanings. The gap $D=b-\tau>0$ is precisely the coefficient of the irreversible term in the exact dissipation law. On a Dirichlet mode with eigenvalue $\lambda_k$, the static coupled form has stiffness matrix
\[
 \begin{pmatrix}\lambda_k&\alpha\\ \alpha&\lambda_k\end{pmatrix},
\]
so $\abs{\alpha}<\lambda_1$ guarantees positivity of every modal stiffness and prevents a coupling-induced static loss of coercivity. Thus the two inequalities in \eqref{eq:parameter-assumptions} are not merely technical smallness conditions: they encode positive relaxation dissipation and static stability of the reciprocal interaction.

The system belongs to the class of indirectly damped hyperbolic models. The wave equation has no damping of its own; all dissipation is produced by the MGT component and can reach $u$ only through the zero-order coupling. This mechanism is substantially weaker than direct friction or structural damping. The general theory of indirect stabilization shows that the transfer of dissipation depends sensitively on the order of the coupling, propagation speeds, and observability properties; see, among many contributions, \cite{AlabauCannarsaKomornik,Alabau1999,Alabau2002,Alabau2003,AlabauLeautaud,BLR}.

Logarithmic nonlinearities were introduced in nonlinear wave mechanics by Bia\l ynicki-Birula and Mycielski \cite{BialynickiMycielski}, while a rigorous evolution-equation framework was developed by Cazenave and Haraux \cite{CazenaveHaraux}. Recent work on logarithmic wave and viscoelastic equations treats local and global existence, decay, potential wells, and finite-time blow-up under strong damping, memory, variable coefficients, and acoustic boundary effects; see \cite{DiShangSong2020,HaPark2020,HaoDu2022,HaApplicable2023,PengZhang2024}. For pure-power sources, potential-well dynamics and extensions to interior or boundary supercritical regimes have likewise remained active; see \cite{LiuLi2020,HaAMO2021,HaEJDE2025} and the references therein. Potential-well methods for nonlinear hyperbolic equations originate in the work of Sattinger \cite{Sattinger} and Payne--Sattinger \cite{PayneSattinger}; complementary concavity mechanisms go back to Levine and subsequent developments \cite{Levine,LevineTodorova,TodorovaVitillaro,Vitillaro}. Related critical-depth and high-energy developments include \cite{GazzolaSquassina,XuCritical}. The present problem requires a modified geometry because the natural static variables are not $(u,v)$ but $(u,v+\tau v_t)$, and because a Levine functional containing a wave-gradient history term is unavailable in the absence of $-\Delta u_t$.

The decisive structural variable is
\[
 w=v+\tau v_t.
\]
With $D=b-\tau>0$ and $z=v_t$, the augmented formulation reveals the coercive coupled form
\[
 Q_\alpha(u,w)=\norm{\nabla u}_2^2+\norm{\nabla w}_2^2+2\alpha\ip{u}{w},
\]
and an exact energy satisfying
\begin{equation}\label{eq:intro-dissipation}
 E(t)+D\int_0^t\norm{\nabla z(s)}_2^2\,\dd s=E(0).
\end{equation}
Thus $(u,w)$, rather than $(u,v)$, is the natural static pair for the potential-well geometry, while the MGT component is the sole source of dissipation.

To treat general focusing sources, we introduce the structural remainder
\[
 H_\theta(s)=\frac1\theta sf(s)-F(s),\qquad \theta>2. 
\]
Its nonnegativity provides coercivity on the stable side, whereas its radial monotonicity yields the Nehari barrier used in the unstable-side concavity argument. This single remainder therefore unifies the potential-well analysis for pure-power, logarithmic, and related focusing sources.

The stable asymptotic problem contains an additional difficulty. The energy identity gives boundedness and finite total MGT dissipation, but it does not compactify the undamped wave component. At the linear level, we prove that the semigroup is strongly stable for nonzero coupling. This stability is nevertheless nonuniform: the wave branch has eigenvalues
\[
 \operatorname{Re}s_{k,\pm}\sim-\frac{\alpha^2}{2(b-\tau)\lambda_k^2},
\]
and the linearized flow is not compact at any fixed positive time; see also \cite{HaLogMGT}. The nonlinear stable analysis therefore proceeds in two stages. First, for nonzero coupling, every stable solution converges weakly to zero without a compactness hypothesis. Second, a renormalized high-frequency balance identifies the exact obstruction to strong compactness: within the stable class, vanishing of the accumulated nonlinear high--low flux is equivalent to relative compactness of the orbit and hence to strong convergence. Concrete sufficient conditions include a spectral finite-variation condition, the coordinate-free hypothesis
\[
 u_t\in L^1(0,\infty;H_{A_D}^{-s}(\Omega))
\]
for a source-dependent range of $s$, finite total variation of the nonlinear force $f(u)$ in $L^2(\Omega)$, and any eventual uniform bound with a positive fractional gain of spatial regularity. These are trajectory-level criteria, and the first-order energy identity alone does not provide the corresponding time-integrability or regularity.

At the regular level, the absence of direct wave damping remains decisive. Unlike a strongly damped wave component, the present equation supplies no a priori $L^2(0,T;H_0^1)$-control of $u_t$ and no positive-time smoothing. We nevertheless prove that regular stable data remain regular on every finite time interval. The resulting high-order estimate is generally exponential in time. If, in addition, the nonlinear force has finite total variation in $L^2(\Omega)$, then subtracting the correction $(f(u),A_Du)$ from the regular energy yields a uniform regular bound and square-integrability of $A_Dz$. Thus the same source-variation condition has two roles: at the energy level it implies the high-frequency flux condition, while for regular trajectories it gives a strictly stronger regularity upgrade.

\medskip
\noindent\textbf{Relation to the logarithmic wave--MGT study.}
The logarithmic specialization of \eqref{eq:original-system} was analyzed in \cite{HaLogMGT}, where the augmented formulation, the exact energy law, the logarithmic stable well, the leading wave-branch spectral obstruction, and convergence under orbit precompactness were established. The present paper takes that structure as its starting point and develops the dynamics generated by the general remainder $H_\theta$: a maximal local theory for arbitrary finite-energy data, strong stability of the linearized semigroup, weak asymptotic stability without compactness, an exact high-frequency characterization of strong convergence, source-variation and regularity upgrades, and the below-depth, critical-depth (for nonzero coupling), and prescribed-energy blow-up mechanisms. Thus \cite{HaLogMGT} is recovered on the common logarithmic stable branch, while Sections~\ref{sec:local}--\ref{sec:positive-blowup} develop the additional linear, asymptotic, regularity, and blow-up results.

The main conclusions are as follows.
\begin{enumerate}[label=\textup{(\roman*)},leftmargin=2.2em]
\item Every finite-energy initial state generates a unique maximal energy solution satisfying the exact identity \eqref{eq:intro-dissipation} and the standard continuation alternative.
\item For $0<\abs{\alpha}<\lambda_1$, the linearized semigroup is strongly stable. The wave-branch expansion rules out uniform exponential stability, and the semigroup is not compact at any positive time.
\item If $E(0)<d$ and $I(u_0,w_0)>0$, the stable set is positively invariant and the solution is global. For $0<\abs{\alpha}<\lambda_1$, every such trajectory converges weakly to zero. The accumulated high-frequency flux condition is equivalent to orbit precompactness and to strong convergence. It follows, in particular, from spectral finite variation, finite variation in a negative Dirichlet scale, or finite total variation of the nonlinear force; eventual positive fractional regularity gives another direct compactness mechanism.
\item If the stable initial state belongs to the regular phase space, the solution remains regular globally, with bounds on every finite time interval. Under finite total variation of $f(u)$ in $L^2(\Omega)$, the regular norm is uniformly bounded on $[0,\infty)$ and $A_Dz\in L^2(0,\infty;L^2(\Omega))$.
\item If $E(0)<d$ and $I(u_0,w_0)<0$, the unstable set is positively invariant and the maximal solution blows up in finite time, with no sign condition on the initial velocities. At the critical level $E(0)=d$, nonzero coupling yields a complete trichotomy into stable entry, finite-time blow-up, or a stationary Nehari ground state.
\item Independently of the Nehari threshold, a data-only outward-projection criterion yields finite-time blow-up at nonnegative energy for arbitrary initial MGT displacement. Every prescribed nonnegative energy shell contains blow-up data.
\item Pure-power, logarithmic, and positive finite sums of powers satisfy the abstract source hypotheses.
\end{enumerate}

The paper is organized as follows. Section~\ref{sec:source} introduces the source class and the augmented formulation. Section~\ref{sec:local} establishes local well-posedness, the exact energy identity, strong stability of the linearized semigroup, and the high-frequency obstruction to uniform exponential stability. Section~\ref{sec:well} develops the coupled potential well. Section~\ref{sec:stable} treats global stable dynamics, weak asymptotic stability, the high-frequency characterization, source-variation criteria, and the regularity upgrade. Section~\ref{sec:below-blowup} proves below-depth blow-up and the critical-level classification, while Section~\ref{sec:positive-blowup} establishes the positive-energy mechanism. Section~\ref{sec:examples} verifies the theory for power and logarithmic sources. Appendix~\ref{app:regular-estimates} contains the regular-phase composition, chain-rule, and Galerkin arguments.

\section{Source structure and augmented formulation}\label{sec:source}
We write $\norm{\cdot}_r$ for the norm of $L^r(\Omega)$ and $\ip{\cdot}{\cdot}$ for the inner product of $L^2(\Omega)$. Let $A_D=-\Delta$ denote the positive Dirichlet Laplacian on $L^2(\Omega)$, with
\[
 D(A_D)=H^2(\Omega)\cap H_0^1(\Omega).
\]
Set
\begin{equation}\label{eq:critical-exponents}
 2^*=\begin{cases}
 \infty,&n=1,2,\\[1mm]
 \dfrac{2n}{n-2},&n\ge3,
 \end{cases}
 \qquad
 p^*=\begin{cases}
 \infty,&n=1,2,\\[1mm]
 \dfrac{2}{n-2},&n\ge3.
 \end{cases}
\end{equation}

\subsection{Hypotheses on the source}
Throughout the general theory we impose the following assumptions.
\begin{description}[leftmargin=3.2em,labelindent=0em]
\item[\textup{(F1)}] $f\in C^1(\R)$, $f(0)=0$, and
\begin{equation}\label{eq:F1-small}
 \lim_{s\to0}\frac{f(s)}s=0.
\end{equation}
Moreover, there exist $C_f>0$ and $p\in(0,p^*)$ such that
\begin{equation}\label{eq:F1-growth}
 \abs{f'(s)}\le C_f\bigl(1+\abs{s}^p\bigr),\qquad s\in\R.
\end{equation}
\item[\textup{(F2)}] There exists $\theta>2$ such that, with
\begin{equation}\label{eq:Htheta-def}
 F(s)=\int_0^s f(r)\,\dd r,
 \qquad H_\theta(s)=\frac1\theta sf(s)-F(s),
\end{equation}
one has
\begin{equation}\label{eq:Htheta-positive}
 H_\theta(s)\ge0,
 \qquad s\in\R.
\end{equation}
\item[\textup{(F3)}] The structural remainder is radially nondecreasing:
\begin{equation}\label{eq:Htheta-radial}
 H_\theta(\lambda s)\le H_\theta(s),
 \qquad s\in\R,
 \quad 0\le\lambda\le1.
\end{equation}
\item[\textup{(F4)}] There exists $\phi\in H_0^1(\Omega)$ such that
\begin{equation}\label{eq:focusing}
 \int_\Omega F(\phi)\,\dd x>0.
\end{equation}
\end{description}
The restriction in \eqref{eq:F1-growth} is the natural one for the $L^2$-based uniqueness argument. For $n\ge3$, it is equivalent to
\begin{equation}\label{eq:subcritical-equivalent}
 2(p+1)<2^*,
\end{equation}
so the Nemytskii map induced by $f$ is locally Lipschitz from $H_0^1(\Omega)$ into $L^2(\Omega)$.

\begin{remark}[Use of the structural assumptions]\label{rem:use-assumptions}
Assumption \textup{(F1)} is sufficient for the local semigroup theory. The well depth and the stable branch use \textup{(F1)}, \textup{(F2)}, and \textup{(F4)}. Radial monotonicity \textup{(F3)} enters through the Nehari barrier for below-depth blow-up and the transversality argument at the critical well depth. The positive-energy criterion requires only \textup{(F1)} and \textup{(F2)}, while the prescribed-level construction additionally uses \textup{(F4)}.
\end{remark}

\begin{remark}\label{rem:sufficient-H}
Under \textup{(F1)}, radial monotonicity \textup{(F3)} is equivalent to
\begin{equation}\label{eq:Htheta-differential}
 s^2f'(s)\ge(\theta-1)sf(s),
 \qquad s\in\R.
\end{equation}
Indeed,
\[
 \frac{\dd}{\dd\lambda}H_\theta(\lambda s)
 =\frac{(\lambda s)^2f'(\lambda s)-(\theta-1)(\lambda s)f(\lambda s)}{\theta\lambda},
 \qquad \lambda>0.
\]
Thus \eqref{eq:Htheta-differential} implies that $\lambda\mapsto H_\theta(\lambda s)$ is nondecreasing. Conversely, \textup{(F3)} makes this map nondecreasing on every radial segment, and its left derivative at $\lambda=1$ yields \eqref{eq:Htheta-differential}. Moreover, since $H_\theta(0)=0$, condition \textup{(F3)} implies \textup{(F2)}. We retain \textup{(F2)} separately because several results below use only the nonnegativity of $H_\theta$, without radial monotonicity. Both the pure-power and logarithmic sources satisfy \eqref{eq:Htheta-differential}; see Section~\ref{sec:examples}.
\end{remark}

\begin{lemma}[Subcritical source estimates]\label{lem:source-estimates}
Assume \textup{(F1)}. For every $\varepsilon>0$ there exists $C_\varepsilon>0$ such that
\begin{align}
 \abs{f(s)}&\le\varepsilon\abs{s}+C_\varepsilon\abs{s}^{p+1},\label{eq:f-eps}\\
 \abs{sf(s)}+\abs{F(s)}&\le\varepsilon\abs{s}^2+C_\varepsilon\abs{s}^{p+2}\label{eq:F-eps}
\end{align}
for all $s\in\R$. Moreover,
\begin{equation}\label{eq:f-difference-pointwise}
 \abs{f(a)-f(b)}\le C\bigl(1+\abs{a}^p+\abs{b}^p\bigr)\abs{a-b},
 \qquad a,b\in\R.
\end{equation}
Consequently, for every $R>0$ there exists $C_R>0$ such that
\begin{equation}\label{eq:f-local-Lipschitz}
 \norm{f(u)-f(v)}_2\le C_R\norm{\nabla(u-v)}_2
\end{equation}
whenever $u,v\in H_0^1(\Omega)$ and $\norm{\nabla u}_2+\norm{\nabla v}_2\le R$.
\end{lemma}

\begin{proof}
The limit \eqref{eq:F1-small} gives $\abs{f(s)}\le\varepsilon\abs{s}$ for $\abs{s}$ sufficiently small. On the complementary region, integration of \eqref{eq:F1-growth} yields
\[
 \abs{f(s)}\le C\bigl(\abs{s}+\abs{s}^{p+1}\bigr),
\]
and the linear term can be absorbed into $C_\varepsilon\abs{s}^{p+1}$ because $p>0$ and $\abs{s}$ is bounded away from zero. This proves \eqref{eq:f-eps}; integration gives the estimate for $F$, and multiplication by $\abs{s}$ gives the estimate for $sf(s)$. The mean-value theorem and \eqref{eq:F1-growth} imply \eqref{eq:f-difference-pointwise}.

For \eqref{eq:f-local-Lipschitz}, use H\"older's inequality with exponents $2(p+1)/p$ and $2(p+1)$, followed by $H_0^1(\Omega)\hookrightarrow L^{2(p+1)}(\Omega)$. The constant term in \eqref{eq:f-difference-pointwise} is handled by Poincar\'e's inequality.
\end{proof}

\begin{lemma}[Differentiability of the source map]\label{lem:nemytskii-C1}
Assume \textup{(F1)}. The Nemytskii operator
\[
 f:H_0^1(\Omega)\longrightarrow L^2(\Omega),
 \qquad u\longmapsto f(u),
\]
is continuously Fr\'echet differentiable, with
\begin{equation}\label{eq:nemytskii-derivative}
 Df(u)h=f'(u)h,
 \qquad u,h\in H_0^1(\Omega).
\end{equation}
\end{lemma}

\begin{proof}
Set $r=2(p+1)$ and $s_p=2(p+1)/p$. By \eqref{eq:subcritical-equivalent},
$H_0^1(\Omega)\hookrightarrow L^r(\Omega)$. H\"older's inequality and
\eqref{eq:F1-growth} give
\[
 \norm{f'(u)h}_2
 \le \norm{f'(u)}_{s_p}\norm{h}_r
 \le C\bigl(1+\norm{u}_r^p\bigr)\norm{h}_r.
\]
Thus $h\mapsto f'(u)h$ belongs to
$\mathcal L(H_0^1(\Omega),L^2(\Omega))$.

Let $u_j\to u$ in $H_0^1(\Omega)$. Then $u_j\to u$ in $L^r(\Omega)$ and,
in particular, in measure. Since $f'$ is continuous,
$f'(u_j)\to f'(u)$ in measure. Moreover, using $ps_p=r$ and
\eqref{eq:F1-growth},
\[
 \abs{f'(u_j)-f'(u)}^{s_p}
 \le C\bigl(1+\abs{u_j}^{r}+\abs{u}^{r}\bigr).
\]
The family on the right-hand side is uniformly integrable. Indeed,
$u_j\to u$ in $L^r(\Omega)$ implies
$\abs{u_j}^{r}\to\abs{u}^{r}$ in $L^1(\Omega)$. Vitali's theorem therefore
yields
\[
 f'(u_j)\to f'(u)
 \qquad\text{strongly in }L^{s_p}(\Omega).
\]
Consequently,
\[
 \norm{(f'(u_j)-f'(u))h}_2
 \le C\norm{f'(u_j)-f'(u)}_{s_p}\norm{h}_{H_0^1},
\]
so the multiplication operators $h\mapsto f'(u_j)h$ converge to
$h\mapsto f'(u)h$ in
$\mathcal L(H_0^1(\Omega),L^2(\Omega))$.

Finally,
\[
 f(u+h)-f(u)-f'(u)h
 =\int_0^1\bigl[f'(u+\sigma h)-f'(u)\bigr]h\,\dd\sigma.
\]
The operator-norm continuity established above gives
\[
 \norm{f(u+h)-f(u)-f'(u)h}_2
 \le \sup_{0\le\sigma\le1}
 \norm{f'(u+\sigma h)-f'(u)}_{s_p}\norm{h}_r
 =o\bigl(\norm{h}_{H_0^1}\bigr).
\]
Hence the Nemytskii map is continuously Fr\'echet differentiable and
\eqref{eq:nemytskii-derivative} holds.
\end{proof}

\begin{lemma}[Scaling of the potential]\label{lem:potential-scaling}
Assume \textup{(F2)}. Then
\begin{equation}\label{eq:F-scaling}
 F(\lambda s)\ge\lambda^\theta F(s),
 \qquad s\in\R,\quad \lambda\ge1.
\end{equation}
In particular,
\begin{equation}\label{eq:AR-like}
 sf(s)\ge\theta F(s),
 \qquad s\in\R.
\end{equation}
\end{lemma}

\begin{proof}
The second assertion is exactly \eqref{eq:Htheta-positive}. For fixed $s\in\R$,
\[
 \frac{\dd}{\dd\lambda}\bigl[\lambda^{-\theta}F(\lambda s)\bigr]
 =\theta\lambda^{-\theta-1}H_\theta(\lambda s)\ge0.
\]
Integration from $1$ to $\lambda\ge1$ gives \eqref{eq:F-scaling}.
\end{proof}

\subsection{The augmented phase space}
Set
\begin{equation}\label{eq:Dzw}
 D=b-\tau>0,
 \qquad z=v_t,
 \qquad w=v+\tau v_t.
\end{equation}
Then \eqref{eq:original-system} is equivalent to
\begin{equation}\label{eq:augmented-system}
\begin{cases}
 u_{tt}-\Delta u+\alpha w=f(u),\\
 w_{tt}-\Delta w-D\Delta z+\alpha u=0,\\
 \tau z_t+z=w_t,
\end{cases}
\end{equation}
with homogeneous Dirichlet conditions for $u,w,z$. The natural phase space is
\begin{equation}\label{eq:phase-space}
 \Hcal=H_0^1(\Omega)\times L^2(\Omega)\times H_0^1(\Omega)\times L^2(\Omega)\times H_0^1(\Omega),
\end{equation}
with state
\begin{equation}\label{eq:state}
 Y(t)=\bigl(u(t),u_t(t),w(t),w_t(t),z(t)\bigr).
\end{equation}
For initial data in the original variables,
\begin{equation}\label{eq:transformed-data}
 w_0=v_0+\tau v_1,
 \qquad w_1=v_1+\tau v_2,
 \qquad z_0=v_1.
\end{equation}
Conversely,
\begin{equation}\label{eq:recover-v}
 v=w-\tau z,
 \qquad v_t=z.
\end{equation}

Define
\begin{equation}\label{eq:Qalpha}
 Q_\alpha(u,w)=\norm{\nabla u}_2^2+\norm{\nabla w}_2^2+2\alpha\ip{u}{w}.
\end{equation}

\begin{lemma}[Coercivity of the coupled form]\label{lem:Q-coercive}
Under \eqref{eq:parameter-assumptions},
\begin{equation}\label{eq:Q-coercive}
 Q_\alpha(u,w)\ge c_\alpha\bigl(\norm{\nabla u}_2^2+\norm{\nabla w}_2^2\bigr),
 \qquad c_\alpha=1-\frac{\abs{\alpha}}{\lambda_1}>0.
\end{equation}
\end{lemma}

\begin{proof}
Young's and Poincar\'e's inequalities give
\[
 2\abs{\alpha}\abs{\ip{u}{w}}
 \le\abs{\alpha}\bigl(\norm{u}_2^2+\norm{w}_2^2\bigr)
 \le\frac{\abs{\alpha}}{\lambda_1}
 \bigl(\norm{\nabla u}_2^2+\norm{\nabla w}_2^2\bigr).
\]
\end{proof}

For $Y=(u,p,w,q,z)\in\Hcal$, define
\begin{equation}\label{eq:phase-energy}
 E(Y)=\frac12\norm{p}_2^2+\frac12\norm{q}_2^2+\frac{\tau D}{2}\norm{\nabla z}_2^2
 +\frac12Q_\alpha(u,w)-\int_\Omega F(u)\,\dd x.
\end{equation}
The static functionals are
\begin{equation}\label{eq:J-I}
 J(u,w)=\frac12Q_\alpha(u,w)-\int_\Omega F(u)\,\dd x,
 \qquad
 I(u,w)=Q_\alpha(u,w)-\int_\Omega uf(u)\,\dd x.
\end{equation}

Set $\Phi(u)=\int_\Omega F(u)\,\dd x$. By Lemma~\ref{lem:nemytskii-C1} and the fundamental theorem of calculus along line segments,
\begin{equation}\label{eq:Phi-derivative}
 \Phi\in C^1\bigl(H_0^1(\Omega);\R\bigr),
 \qquad \Phi'(u)h=\ip{f(u)}{h}.
\end{equation}
The continuity of the $L^2$-pairing then gives
\[
 J,I\in C^1\bigl(H_0^1(\Omega)^2;\R\bigr).
\]
More precisely, for $(u,w),(h,k)\in H_0^1(\Omega)^2$,
\begin{align}
 \dual{J'(u,w)}{(h,k)}
 ={}&\ip{\nabla u}{\nabla h}+\ip{\nabla w}{\nabla k}
 +\alpha\bigl(\ip{h}{w}+\ip{u}{k}\bigr)-\ip{f(u)}{h},
 \label{eq:J-derivative}\\
 \dual{I'(u,w)}{(h,k)}
 ={}&2\ip{\nabla u}{\nabla h}+2\ip{\nabla w}{\nabla k}
 +2\alpha\bigl(\ip{h}{w}+\ip{u}{k}\bigr)\notag\\
 &-\int_\Omega\bigl(f(u)+uf'(u)\bigr)h\,\dd x.
 \label{eq:I-derivative}
\end{align}
In particular,
\begin{equation}\label{eq:J-radial-derivative}
 \dual{J'(u,w)}{(u,w)}=I(u,w).
\end{equation}

\begin{lemma}[Structural identity]\label{lem:structural-identity}
For every $(u,w)\in H_0^1(\Omega)\times H_0^1(\Omega)$,
\begin{equation}\label{eq:structural-identity}
 J(u,w)=\frac{\theta-2}{2\theta}Q_\alpha(u,w)+\frac1\theta I(u,w)
 +\int_\Omega H_\theta(u)\,\dd x.
\end{equation}
\end{lemma}

\begin{proof}
Use $F(s)=\theta^{-1}sf(s)-H_\theta(s)$ in \eqref{eq:J-I} and collect terms.
\end{proof}

\section{Local well-posedness, exact energy identity, and linear stability}\label{sec:local}
We first identify the linear augmented dynamics as a dissipative semigroup. On $\Hcal$, introduce
\begin{align}\label{eq:alpha-inner-product}
 \dual{Y}{\widetilde Y}_\alpha
 :={}&\ip{p}{\widetilde p}+\ip{q}{\widetilde q}
 +\tau D\ip{\nabla z}{\nabla\widetilde z}
 +\ip{\nabla u}{\nabla\widetilde u}+\ip{\nabla w}{\nabla\widetilde w}\notag\\
 &+\alpha\bigl(\ip{u}{\widetilde w}+\ip{w}{\widetilde u}\bigr),
\end{align}
where $Y=(u,p,w,q,z)$ and $\widetilde Y=(\widetilde u,\widetilde p,\widetilde w,\widetilde q,\widetilde z)$. Lemma~\ref{lem:Q-coercive} implies that \eqref{eq:alpha-inner-product} is an inner product whose norm is equivalent to the standard norm of $\Hcal$.

Define $\Acal:D(\Acal)\subset\Hcal\to\Hcal$ by
\begin{equation}\label{eq:linear-generator}
 \Acal
 \begin{pmatrix}u\\p\\w\\q\\z\end{pmatrix}
 =
 \begin{pmatrix}
 p\\
 \Delta u-\alpha w\\
 q\\
 \Delta(w+Dz)-\alpha u\\
 \tau^{-1}(q-z)
 \end{pmatrix},
\end{equation}
with
\begin{align}\label{eq:generator-domain}
 D(\Acal)=\bigl\{&(u,p,w,q,z)\in\Hcal:
 p,q\in H_0^1(\Omega),\ u\in H^2(\Omega)\cap H_0^1(\Omega),\notag\\
 &\hspace{37mm}w+Dz\in H^2(\Omega)\cap H_0^1(\Omega)\bigr\}.
\end{align}
Equivalently, one may require $\Delta u,\Delta(w+Dz)\in L^2(\Omega)$ in the distributional sense. The domain condition is imposed on the combination $w+Dz$, rather than separately on $w$ and $z$, because the fourth component of $\Acal Y$ contains only $\Delta(w+Dz)$. Thus no separate $H^2$-regularity of $w$ or $z$ is required at the level of the generator domain.

\begin{proposition}[Linear semigroup realization]\label{prop:linear-semigroup}
Assume \eqref{eq:parameter-assumptions}. The operator $\Acal$ is maximal dissipative on $(\Hcal,\dual{\cdot}{\cdot}_\alpha)$ and generates a contraction $C_0$-semigroup $\{e^{t\Acal}\}_{t\ge0}$. Moreover,
\begin{equation}\label{eq:generator-dissipative}
 \dual{\Acal Y}{Y}_\alpha=-D\norm{\nabla z}_2^2,
 \qquad Y\in D(\Acal).
\end{equation}
\end{proposition}

\begin{proof}
The operator is the linear augmented generator used in \cite{HaLogMGT}; we recall the short generation argument in the present notation. Density of $D(\Acal)$ follows from Dirichlet spectral truncation applied to $u$, $z$, the combination $w+Dz$, and the velocity components. Integration by parts gives \eqref{eq:generator-dissipative}, so $\Acal$ is dissipative.

Fix $\lambda>0$ and $G=(g_1,g_2,g_3,g_4,g_5)\in\Hcal$. Solving $(\lambda I-\Acal)Y=G$ for $p,q,z$ gives
\begin{equation}\label{eq:resolvent-pqz}
 p=\lambda u-g_1,
 \qquad q=\lambda w-g_3,
 \qquad z=c_\lambda w+h_\lambda,
\end{equation}
where
\begin{equation}\label{eq:clambda-hlambda}
 c_\lambda=\frac{\lambda}{1+\tau\lambda},
 \qquad h_\lambda=\frac{\tau g_5-g_3}{1+\tau\lambda}\in H_0^1(\Omega).
\end{equation}
The remaining equations are equivalent to
\begin{align}\label{eq:resolvent-variational}
 &\ip{\nabla u}{\nabla\varphi}+(1+Dc_\lambda)\ip{\nabla w}{\nabla\psi}
 +\lambda^2\ip{u}{\varphi}+\lambda^2\ip{w}{\psi}
 +\alpha\bigl(\ip{w}{\varphi}+\ip{u}{\psi}\bigr)\notag\\
 &\qquad=\ip{g_2+\lambda g_1}{\varphi}+\ip{g_4+\lambda g_3}{\psi}
 -D\ip{\nabla h_\lambda}{\nabla\psi}
\end{align}
for all $(\varphi,\psi)\in H_0^1(\Omega)^2$. Since $1+Dc_\lambda\ge1$, Lemma~\ref{lem:Q-coercive} makes the left-hand side coercive. Lax--Milgram therefore gives a unique $(u,w)$; the resolvent equations yield $\Delta u,\Delta(w+Dz)\in L^2(\Omega)$ and hence $Y\in D(\Acal)$. Thus $\operatorname{Ran}(\lambda I-\Acal)=\Hcal$, and the Lumer--Phillips theorem completes the proof.
\end{proof}

Let
\begin{equation}\label{eq:nonlinear-G}
 \mathcal G(u,p,w,q,z)=(0,f(u),0,0,0).
\end{equation}
Lemmas~\ref{lem:source-estimates} and \ref{lem:nemytskii-C1} show that $\mathcal G:\Hcal\to\Hcal$ is locally Lipschitz and continuously Fr\'echet differentiable.

\begin{definition}[Energy solution]\label{def:energy-solution}
Let $T>0$ and $Y_0\in\Hcal$. An energy solution of \eqref{eq:augmented-system} on $[0,T]$ is a function $Y\in C([0,T];\Hcal)$ satisfying
\begin{equation}\label{eq:variation-constants}
 Y(t)=e^{t\Acal}Y_0+\int_0^t e^{(t-s)\Acal}\mathcal G(Y(s))\,\dd s,
 \qquad 0\le t\le T.
\end{equation}
\end{definition}
Every energy solution satisfies \eqref{eq:augmented-system} in distributions and
\begin{align}\label{eq:time-regularity}
 u,w&\in C([0,T];H_0^1)\cap C^1([0,T];L^2)\cap C^2([0,T];H^{-1}),\notag\\
 z&\in C([0,T];H_0^1)\cap C^1([0,T];L^2).
\end{align}

\begin{theorem}[Local well-posedness and continuation]\label{thm:local-wp}
Assume \eqref{eq:parameter-assumptions} and \textup{(F1)}. For every $Y_0\in\Hcal$, there exists a unique maximal energy solution on $[0,T_{\max})$, where $0<T_{\max}\le\infty$. The solution depends continuously on its initial state on every common compact interval of existence. Moreover, for every $0\le s\le t<T_{\max}$,
\begin{equation}\label{eq:exact-energy}
 E(t)+D\int_s^t\norm{\nabla z(\xi)}_2^2\,\dd\xi=E(s),
 \qquad E(t)=E(Y(t)),
\end{equation}
and
\begin{equation}\label{eq:continuation-alternative}
 T_{\max}<\infty
 \quad\Longrightarrow\quad
 \limsup_{t\uparrow T_{\max}}\norm{Y(t)}_{\Hcal}=\infty.
\end{equation}
In particular, $E\in C^1([0,T_{\max}))$ and
\begin{equation}\label{eq:Eprime}
 E'(t)=-D\norm{\nabla z(t)}_2^2.
\end{equation}
\end{theorem}

\begin{proof}
Proposition~\ref{prop:linear-semigroup}, local Lipschitz continuity of $\mathcal G$, and the standard semilinear $C_0$-semigroup theorem yield a unique maximal mild solution, continuous dependence, and \eqref{eq:continuation-alternative}; see \cite[Chapter~6]{Pazy}.

Let $\Phi$ be the potential functional introduced after \eqref{eq:J-I}; its differentiability is recorded in \eqref{eq:Phi-derivative}. Choose $Y_{0m}\in D(\Acal)$ with $Y_{0m}\to Y_0$ in $\Hcal$, and fix $[0,T]\Subset[0,T_{\max})$. Continuous dependence for the semilinear problem gives a common existence interval $[0,T]$ for all sufficiently large $m$ and
\begin{equation}\label{eq:classical-approximation}
 Y_m\to Y\qquad\text{in }C([0,T];\Hcal).
\end{equation}
Because $\mathcal G$ is continuously Fr\'echet differentiable, the solutions issued from $D(\Acal)$ are classical on $[0,T]$. For such a solution, \eqref{eq:generator-dissipative}, \eqref{eq:Phi-derivative}, and $u_{m,t}=p_m$ give
\[
 \frac{\dd}{\dd t}\left(\frac12\norm{Y_m(t)}_\alpha^2-\Phi(u_m(t))\right)
 =-D\norm{\nabla z_m(t)}_2^2.
\]
The expression in parentheses is $E_m(t)$. Integrating from $s$ to $t$ gives the exact identity at the classical level. By \eqref{eq:classical-approximation}, the quadratic energies converge uniformly on $[0,T]$; moreover, $u_m\to u$ in $C([0,T];H_0^1)$ and the continuity of $\Phi$ imply
\[
 \Phi(u_m)\to\Phi(u)\qquad\text{in }C([0,T]).
\]
Likewise, $z_m\to z$ in $C([0,T];H_0^1)$, and therefore
\[
 \int_s^t\norm{\nabla z_m(\xi)}_2^2\,\dd\xi
 \longrightarrow
 \int_s^t\norm{\nabla z(\xi)}_2^2\,\dd\xi
\]
uniformly for $(s,t)$ in the closed triangle $0\le s\le t\le T$. Passing to the limit proves \eqref{eq:exact-energy}. Finally,
\[
 E(t)=E(0)-D\int_0^t\norm{\nabla z(\xi)}_2^2\,\dd\xi,
\]
and $z\in C([0,T];H_0^1)$; hence $E\in C^1([0,T])$ and \eqref{eq:Eprime} follows. Since $T<T_{\max}$ was arbitrary, the conclusion holds on the maximal interval.
\end{proof}

\begin{remark}[Recovery of the original MGT variable]\label{rem:recover-v}
Let $Y=(u,u_t,w,w_t,z)$ be the maximal energy solution and set $v=w-\tau z$. The third equation in \eqref{eq:augmented-system} gives $v_t=z$. In fact,
\begin{equation}\label{eq:v-regularity}
 v\in C^1([0,T];H_0^1(\Omega))\cap C^2([0,T];L^2(\Omega)),
 \qquad v_t=z,
 \qquad v_{tt}=z_t.
\end{equation}
Consequently, $(u,v)$ solves \eqref{eq:original-system} in distributions. For $0\le s\le t<T_{\max}$,
\begin{equation}\label{eq:v-gradient-identity}
 \norm{\nabla v(t)}_2^2-\norm{\nabla v(s)}_2^2
 =2\int_s^t\ip{\nabla v(\xi)}{\nabla z(\xi)}\,\dd\xi.
\end{equation}
\end{remark}

\subsection{Strong stability and the high-frequency obstruction}
Under \textup{(F1)}, $f'(0)=0$, so the linearization at the origin is
\begin{equation}\label{eq:linearized-original}
\begin{cases}
 u_{tt}-\Delta u+\alpha(v+\tau v_t)=0,\\
 \tau v_{ttt}+v_{tt}-\Delta v-b\Delta v_t+\alpha u=0.
\end{cases}
\end{equation}
Let $-\Delta e_k=\lambda_k e_k$ be the Dirichlet eigenpairs. Modal solutions $u=a_ke^{st}e_k$, $v=c_ke^{st}e_k$ exist precisely when
\begin{equation}\label{eq:Pk}
 P_k(s)=\bigl(s^2+\lambda_k\bigr)\bigl(\tau s^3+s^2+b\lambda_k s+\lambda_k\bigr)
 -\alpha^2(1+\tau s)=0.
\end{equation}
The polynomial \eqref{eq:Pk} is independent of the higher-order structure of the source. We recall the wave-branch calculation from \cite{HaLogMGT} and retain one further order in its remainder.

\begin{proposition}[Wave-branch spectral obstruction]\label{prop:spectral-obstruction}
Assume $b>\tau>0$ and $\alpha\ne0$. For each $\sigma\in\{+1,-1\}$ and all sufficiently large $k$, $P_k$ has a root $s_{k,\sigma}$ satisfying
\begin{equation}\label{eq:spectral-expansion}
 s_{k,\sigma}
 =\sigma i\sqrt{\lambda_k}
 -\sigma i\frac{\alpha^2\tau}{2(b-\tau)\lambda_k^{3/2}}
 -\frac{\alpha^2}{2(b-\tau)\lambda_k^2}
 +O(\lambda_k^{-7/2}).
\end{equation}
In particular,
\begin{equation}\label{eq:real-part-expansion}
 \operatorname{Re}s_{k,\sigma}
 =-\frac{\alpha^2}{2(b-\tau)\lambda_k^2}+O(\lambda_k^{-7/2})\longrightarrow0^-.
\end{equation}
These roots belong to the point spectrum of the complexified generator. Consequently, the linearized semigroup is not uniformly exponentially stable in the natural energy space.
\end{proposition}

\begin{proof}
The modal determinant and the expansion with remainder $O(\lambda_k^{-5/2})$ are established in \cite{HaLogMGT}. To obtain the stated refinement, write $\omega_k=\sqrt{\lambda_k}$ and $s_{0,k}=\sigma i\omega_k$. The exact Newton correction satisfies
\[
 -\frac{P_k(s_{0,k})}{P_k'(s_{0,k})}
 =-\sigma i\frac{\alpha^2\tau}{2(b-\tau)\omega_k^3}
 -\frac{\alpha^2}{2(b-\tau)\omega_k^4}+O(\omega_k^{-7}).
\]
Moreover, at the corrected center $\widehat s_{k,\sigma} =s_{0,k}-\frac{P_k(s_{0,k})}{P_k'(s_{0,k})}$ one has
$P_k(\widehat s_{k,\sigma})=O(\omega_k^{-3})$,
$P_k'(\widehat s_{k,\sigma})=-2(b-\tau)\omega_k^4+O(1)$, and
$P_k''(s)=O(\omega_k^3)$ for $\abs{s-s_{0,k}}\le C\omega_k^{-3}$.
Choosing $R>0$ sufficiently large and then $k$ sufficiently large, Rouch\'e's theorem on the circle $\abs{s-\widehat s_{k,\sigma}}=R\omega_k^{-7}$ gives an actual zero at distance $O(\omega_k^{-7})$ from $\widehat s_{k,\sigma}$. The point-spectrum realization is the same as in \cite{HaLogMGT}, and \eqref{eq:real-part-expansion} excludes a uniform spectral gap.
\end{proof}

\begin{proposition}[Strong but nonuniform linear stability]\label{prop:linear-strong-stability}
Assume \eqref{eq:parameter-assumptions} and $\alpha\ne0$. Then the linear semigroup generated by $\Acal$ is strongly stable:
\begin{equation}\label{eq:linear-strong-stability}
 e^{t\Acal}Y_0\longrightarrow0
 \qquad\text{strongly in }\Hcal\quad\text{as }t\to\infty
\end{equation}
for every $Y_0\in\Hcal$. By Proposition~\ref{prop:spectral-obstruction}, this convergence is not uniformly exponential.
\end{proposition}

\begin{proof}
Write $T(t)=e^{t\Acal}$ and let $P_N$ be the $L^2$-orthogonal projection onto the first $N$ Dirichlet eigenfunctions; set $Q_N=I-P_N$, and denote their componentwise extensions to $\Hcal$ by $\mathbf P_N$ and $\mathbf Q_N$. Since $P_N(D(A_D))\subset D(A_D)$ and $P_NA_D=A_DP_N$ on $D(A_D)$, the definition of $D(\Acal)$ gives
\[
 \mathbf P_ND(\Acal)\subset D(\Acal),
 \qquad
 \mathbf P_N\Acal Y=\Acal\mathbf P_NY,
 \qquad Y\in D(\Acal).
\]
Hence $\mathbf P_N$ commutes with the resolvent of $\Acal$ and with $T(t)$; the same is true of $\mathbf Q_N$. For
\[
 T(t)Y_0=(u(t),u_t(t),w(t),w_t(t),z(t)),
\]
define the tail energy
\begin{align*}
 E_N^{\rm lin}(t)=\frac12\Big(&\norm{Q_Nu_t(t)}_2^2+\norm{Q_Nw_t(t)}_2^2
 +\tau D\norm{\nabla Q_Nz(t)}_2^2\\
 &+Q_\alpha(Q_Nu(t),Q_Nw(t))\Big).
\end{align*}
The linear energy identity, first for data in $D(\Acal)$ and then by density, gives
\begin{equation}\label{eq:linear-tail-identity}
 E_N^{\rm lin}(t)+D\int_0^t\norm{\nabla Q_Nz(s)}_2^2\,\dd s
 =E_N^{\rm lin}(0).
\end{equation}
Coercivity of $Q_\alpha$ therefore yields
\begin{equation}\label{eq:linear-uniform-tail}
 \sup_{t\ge0}\norm{\mathbf Q_NT(t)Y_0}_{\Hcal}^2
 \le C E_N^{\rm lin}(0)\longrightarrow0
 \qquad(N\to\infty).
\end{equation}
For fixed $N$, the set $\{\mathbf P_NT(t)Y_0:t\ge0\}$ is bounded in a finite-dimensional space. Hence \eqref{eq:linear-uniform-tail} implies that the full orbit $\{T(t)Y_0:t\ge0\}$ is relatively compact in $\Hcal$.

Let
\[
 \mathfrak E_{\rm lin}(Y)=\frac12\norm{Y}_\alpha^2,
 \qquad Y\in\Hcal.
\]
Then $t\mapsto\mathfrak E_{\rm lin}(T(t)Y_0)$ is nonincreasing and converges to some $\ell\ge0$. Let $\overline Y$ be an element of the $\omega$-limit set of the orbit, and choose $t_j\to\infty$ such that $T(t_j)Y_0\to\overline Y$ in $\Hcal$. For every $s\ge0$, continuity of the semigroup and of the quadratic energy gives
\[
 \mathfrak E_{\rm lin}(T(s)\overline Y)
 =\lim_{j\to\infty}\mathfrak E_{\rm lin}(T(t_j+s)Y_0)=\ell.
\]
The energy identity along the trajectory issued from $\overline Y$ therefore implies that its fifth component satisfies $z(s)\equiv0$. The relation $\tau z_t+z=w_t$ gives $w_t\equiv0$. The second linear equation then reduces to
\[
 A_Dw+\alpha u=0.
\]
Since $\alpha\ne0$, the function $u$ is also independent of time, and the first equation gives
\[
 A_Du+\alpha w=0.
\]
Testing these two identities by $w$ and $u$, respectively, and adding yields $Q_\alpha(u,w)=0$. Lemma~\ref{lem:Q-coercive} gives $u=w=0$, and consequently all five components vanish. Thus the $\omega$-limit set consists only of the origin. Relative compactness of the orbit now implies \eqref{eq:linear-strong-stability}.
\end{proof}

\begin{remark}[Absence of compactifying linear regularization]\label{rem:no-linear-compactification}
The same modal family shows that the linearized semigroup is not compact at any positive time. Normalize eigenvectors $Y_{k,\sigma}$ in mutually orthogonal Dirichlet modal subspaces. For every fixed $t>0$,
\[
 \norm{e^{t\Acal_\mathbb C}Y_{k,\sigma}}_{\Hcal}
 =e^{t\operatorname{Re}s_{k,\sigma}}\longrightarrow1.
\]
Thus the image of the bounded sequence $\{Y_{k,\sigma}\}$ has no convergent subsequence. The compactness criteria developed in Section~\ref{sec:stable} are therefore necessarily orbitwise; they do not arise from positive-time smoothing of the linear flow.
\end{remark}

\section{The coupled potential well}\label{sec:well}
Define the Nehari manifold and the well depth by
\begin{equation}\label{eq:Nehari-depth}
 \Ncal=\bigl\{(u,w)\in H_0^1(\Omega)^2\setminus\{(0,0)\}:I(u,w)=0\bigr\},
 \qquad d=\inf_{(u,w)\in\Ncal}J(u,w).
\end{equation}
The stable and unstable regions below $d$ are
\begin{align}
 \Wcal&=\bigl\{(u,w):J(u,w)<d,\ I(u,w)>0\bigr\}\cup\{(0,0)\},\label{eq:stable-well}\\
 \Ucal&=\bigl\{(u,w):J(u,w)<d,\ I(u,w)<0\bigr\}.\label{eq:unstable-well}
\end{align}

\begin{proposition}[Positivity of the well depth]\label{prop:positive-depth}
Assume \textup{(F1)}, \textup{(F2)}, and \textup{(F4)}. Then $\Ncal\ne\varnothing$ and $d>0$. More precisely, there exists $\kappa_0>0$ such that
\begin{equation}\label{eq:kappa0}
 Q_\alpha(u,w)\ge\kappa_0,
 \qquad (u,w)\in\Ncal.
\end{equation}
\end{proposition}

\begin{proof}
Let $\phi\in H_0^1(\Omega)$ satisfy \eqref{eq:focusing} and set $g(\lambda)=I(\lambda\phi,0)$. Lemma~\ref{lem:source-estimates} implies $g(\lambda)>0$ for sufficiently small $\lambda>0$. By Lemma~\ref{lem:potential-scaling}, for $\lambda\ge1$,
\[
 \int_\Omega \lambda\phi f(\lambda\phi)\,\dd x
 \ge\theta\int_\Omega F(\lambda\phi)\,\dd x
 \ge\theta\lambda^\theta\int_\Omega F(\phi)\,\dd x.
\]
Since $\theta>2$, $g(\lambda)<0$ for large $\lambda$. Continuity gives a root, so $\Ncal\ne\varnothing$.

Let $(u,w)\in\Ncal$. Then
\[
 Q_\alpha(u,w)=\int_\Omega uf(u)\,\dd x.
\]
By \eqref{eq:F-eps}, Sobolev embedding, and Lemma~\ref{lem:Q-coercive},
\[
 Q_\alpha(u,w)
 \le\frac{\varepsilon}{\lambda_1c_\alpha}Q_\alpha(u,w)
 +C_\varepsilon Q_\alpha(u,w)^{(p+2)/2}.
\]
Choose $\varepsilon$ so that the first coefficient is at most $1/2$. Since $p>0$ and $Q_\alpha(u,w)>0$, a uniform lower bound \eqref{eq:kappa0} follows. Finally, Lemma~\ref{lem:structural-identity}, $I=0$, and $H_\theta\ge0$ give
\[
 J(u,w)\ge\frac{\theta-2}{2\theta}Q_\alpha(u,w)
 \ge\frac{\theta-2}{2\theta}\kappa_0.
\]
Taking the infimum proves $d>0$.
\end{proof}

\begin{lemma}[Local positivity of the Nehari functional]\label{lem:local-Nehari-positive}
Assume \textup{(F1)}. There exists $\rho_0>0$ such that
\begin{equation}\label{eq:local-Nehari-positive}
 Q_\alpha(u,w)\le\rho_0
 \quad\Longrightarrow\quad
 I(u,w)\ge\frac14Q_\alpha(u,w).
\end{equation}
In particular, $I(u,w)>0$ for every nonzero pair sufficiently close to the origin in $H_0^1(\Omega)^2$.
\end{lemma}

\begin{proof}
By \eqref{eq:F-eps}, Sobolev embedding, and \eqref{eq:Q-coercive},
\[
 \int_\Omega uf(u)\,\dd x
 \le\frac{\varepsilon}{\lambda_1c_\alpha}Q_\alpha(u,w)
 +C_\varepsilon Q_\alpha(u,w)^{(p+2)/2}.
\]
Choose $\varepsilon$ so that the first coefficient is at most $1/4$, and then choose $\rho_0$ so that $C_\varepsilon Q^{p/2}\le1/2$ for $0\le Q\le\rho_0$.
\end{proof}

\begin{lemma}[Nehari barrier in the unstable region]\label{lem:Nehari-barrier}
Assume \textup{(F1)}--\textup{(F4)}. If $I(u,w)<0$, then there exists $\lambda_{u,w}\in(0,1)$ such that
\begin{equation}\label{eq:scaling-to-Nehari}
 I(\lambda_{u,w}u,\lambda_{u,w}w)=0.
\end{equation}
Moreover,
\begin{equation}\label{eq:Nehari-barrier}
 (\theta-2)Q_\alpha(u,w)+2\theta\int_\Omega H_\theta(u)\,\dd x\ge2\theta d,
\end{equation}
and
\begin{equation}\label{eq:unstable-Q-lower}
 Q_\alpha(u,w)\ge\kappa_0.
\end{equation}
\end{lemma}

\begin{proof}
Because $I(u,w)<0$, the pair is nonzero. Lemma~\ref{lem:local-Nehari-positive} gives $I(\lambda u,\lambda w)>0$ for small $\lambda>0$, and continuity yields \eqref{eq:scaling-to-Nehari}. By the definition of $d$ and Lemma~\ref{lem:structural-identity},
\[
 d\le J(\lambda_{u,w}u,\lambda_{u,w}w)
 =\frac{\theta-2}{2\theta}\lambda_{u,w}^2Q_\alpha(u,w)
 +\int_\Omega H_\theta(\lambda_{u,w}u)\,\dd x.
\]
Since $0<\lambda_{u,w}<1$, assumption \textup{(F3)} implies \eqref{eq:Nehari-barrier}. Finally,
\[
 \lambda_{u,w}^2Q_\alpha(u,w)
 =Q_\alpha(\lambda_{u,w}u,\lambda_{u,w}w)\ge\kappa_0,
\]
which gives \eqref{eq:unstable-Q-lower}.
\end{proof}

\begin{lemma}[Transversality of the Nehari manifold]\label{lem:Nehari-transversality}
Assume \textup{(F1)}--\textup{(F3)}. For every $(u,w)\in\Ncal$,
\begin{equation}\label{eq:Nehari-transversality}
 \dual{I'(u,w)}{(u,w)}
 \le-(\theta-2)Q_\alpha(u,w)<0.
\end{equation}
Consequently, $\Ncal$ is a $C^1$ codimension-one manifold near each of its points, and every minimizer of $J$ on $\Ncal$ is a critical point of $J$.
\end{lemma}

\begin{proof}
By Remark~\ref{rem:sufficient-H}, assumption \textup{(F3)} implies
\begin{equation}\label{eq:Htheta-pointwise-derivative}
 s^2f'(s)\ge(\theta-1)sf(s),
 \qquad s\in\R.
\end{equation}
Hence, on $\Ncal$,
\begin{align*}
 \dual{I'(u,w)}{(u,w)}
 &=2Q_\alpha(u,w)-\int_\Omega\bigl(uf(u)+u^2f'(u)\bigr)\,\dd x\\
 &\le2Q_\alpha(u,w)-\theta\int_\Omega uf(u)\,\dd x
 =-(\theta-2)Q_\alpha(u,w),
\end{align*}
which proves \eqref{eq:Nehari-transversality}. In particular, $I'(u,w)\ne0$ on $\Ncal$, so the implicit-function theorem gives the local manifold property. If $(u,w)$ minimizes $J$ on $\Ncal$, the Lagrange-multiplier rule gives $J'(u,w)=\mu I'(u,w)$. Pairing with $(u,w)$ and using $\dual{J'(u,w)}{(u,w)}=I(u,w)=0$ together with \eqref{eq:Nehari-transversality} yields $\mu=0$.
\end{proof}

\begin{remark}
The roles of $H_\theta$ are now transparent. In the stable region, \eqref{eq:structural-identity} and $H_\theta\ge0$ make $J$ coercive. In the unstable region, radial monotonicity gives the barrier \eqref{eq:Nehari-barrier}; at the critical level, it yields the transversality \eqref{eq:Nehari-transversality}. The same source remainder therefore controls the stable, unstable, and threshold dynamics.
\end{remark}

\section{Global and asymptotic dynamics in the stable set}\label{sec:stable}
We begin with global existence below the well depth.

\begin{theorem}[Global solutions below the well depth]\label{thm:stable-global}
Assume \eqref{eq:parameter-assumptions} and \textup{(F1)}, \textup{(F2)}, \textup{(F4)}. Let $Y_0=(u_0,u_1,w_0,w_1,z_0)\in\Hcal$ satisfy
\begin{equation}\label{eq:stable-initial}
 E(0)<d,
 \qquad (u_0,w_0)\in\Wcal.
\end{equation}
Then the maximal solution from Theorem~\ref{thm:local-wp} is global. Moreover,
\begin{equation}\label{eq:stable-invariance}
 (u(t),w(t))\in\Wcal,
 \qquad t\ge0,
\end{equation}
and there exists $C=C(Y_0)>0$ such that
\begin{equation}\label{eq:stable-uniform-bound}
 \norm{u_t(t)}_2^2+\norm{w_t(t)}_2^2+\norm{\nabla z(t)}_2^2
 +\norm{\nabla u(t)}_2^2+\norm{\nabla w(t)}_2^2\le C,
 \qquad t\ge0.
\end{equation}
The exact identity \eqref{eq:exact-energy} holds for all $0\le s\le t<\infty$.
\end{theorem}

\begin{proof}
Let $[0,T_{\max})$ be the maximal interval. The maps $t\mapsto I(u(t),w(t))$ and $t\mapsto J(u(t),w(t))$ are continuous. The energy identity gives
\begin{equation}\label{eq:J-below-d}
 J(u(t),w(t))\le E(t)\le E(0)<d,
 \qquad 0\le t<T_{\max}.
\end{equation}
We claim that $I(u(t),w(t))\ge0$ throughout the maximal interval. Local positivity, together with $(u_0,w_0)\in\Wcal$, first gives this inequality on a nontrivial interval, including the case $(u_0,w_0)=(0,0)$. If the open set
\[
 \Sigma=\{t\in(0,T_{\max}):I(u(t),w(t))<0\}
\]
were nonempty, let $t_*=\inf\Sigma$ and choose $t_j\in\Sigma$ with $t_j\downarrow t_*$. Continuity gives $I(u(t_*),w(t_*))=0$. If the pair is nonzero, it belongs to $\Ncal$, so $J(u(t_*),w(t_*))\ge d$, contradicting \eqref{eq:J-below-d}. If it is zero, continuity of $Q_\alpha$ and Lemma~\ref{lem:local-Nehari-positive} give $I\ge0$ near $t_*$, contradicting $I(u(t_j),w(t_j))<0$. Thus $I\ge0$. Equality can occur only at $(u,w)=(0,0)$, again by \eqref{eq:J-below-d}, and \eqref{eq:stable-invariance} follows.

By Lemmas~\ref{lem:structural-identity} and \ref{lem:Q-coercive},
\begin{equation}\label{eq:J-stable-coercive}
 J(u(t),w(t))
 \ge\frac{\theta-2}{2\theta}Q_\alpha(u(t),w(t))
 \ge\frac{\theta-2}{2\theta}c_\alpha
 \bigl(\norm{\nabla u(t)}_2^2+\norm{\nabla w(t)}_2^2\bigr).
\end{equation}
Since
\[
 E(t)=\frac12\norm{u_t(t)}_2^2+\frac12\norm{w_t(t)}_2^2
 +\frac{\tau D}{2}\norm{\nabla z(t)}_2^2+J(u(t),w(t))\le E(0),
\]
all terms are nonnegative and \eqref{eq:stable-uniform-bound} follows. The continuation alternative forces $T_{\max}=\infty$.
\end{proof}

\begin{remark}\label{rem:energy-coercive-stable}
Along every solution in Theorem~\ref{thm:stable-global},
\begin{equation}\label{eq:energy-phase-coercive}
 E(t)\ge c\norm{Y(t)}_{\Hcal}^2
\end{equation}
for some $c>0$ depending only on the structural parameters. In particular, $E$ is nonnegative and has a finite limit as $t\to\infty$.
\end{remark}

Define the invariant phase region
\begin{equation}\label{eq:X-phase-region}
 \Xcal=\bigl\{Y=(u,p,w,q,z)\in\Hcal:E(Y)<d,\ (u,w)\in\Wcal\bigr\}.
\end{equation}

\begin{proposition}[Semiflow and Lyapunov structure]\label{prop:semiflow}
Theorem~\ref{thm:stable-global} defines a continuous semiflow $S(t)$ on $\Xcal$. For every $Y_0\in\Xcal$,
\begin{equation}\label{eq:semiflow-energy}
 E(S(t)Y_0)+D\int_0^t\norm{\nabla z(s)}_2^2\,\dd s=E(Y_0).
\end{equation}
\end{proposition}

\begin{proof}
Existence and positive invariance follow from Theorem~\ref{thm:stable-global}; uniqueness gives the semigroup property, and Theorem~\ref{thm:local-wp} gives continuity with respect to the initial state. The Lyapunov identity is \eqref{eq:exact-energy}.
\end{proof}

\subsection{Weak asymptotic stability without compactness assumptions}
The following theorem is the unconditional asymptotic result that holds at the natural energy level.

\begin{theorem}[Weak asymptotic stability]\label{thm:weak-asymptotic}
Assume the hypotheses of Theorem~\ref{thm:stable-global} and
\begin{equation}\label{eq:nonzero-coupling}
 0<\abs{\alpha}<\lambda_1.
\end{equation}
Then
\begin{equation}\label{eq:weak-convergence-H}
 Y(t)\rightharpoonup0
 \qquad\text{weakly in }\Hcal\quad\text{as }t\to\infty.
\end{equation}
Moreover, for every $2\le r<2^*$,
\begin{equation}\label{eq:strong-lower-convergence}
 u(t)\to0,
 \qquad w(t)\to0,
 \qquad z(t)\to0
 \quad\text{strongly in }L^r(\Omega).
\end{equation}
In particular,
\begin{equation}\label{eq:potential-to-zero}
 \int_\Omega F(u(t))\,\dd x\to0,
 \qquad \ip{u(t)}{w(t)}_{L^2(\Omega)}\to0.
\end{equation}
\end{theorem}

\begin{proof}
By Remark~\ref{rem:energy-coercive-stable}, the orbit is bounded in $\Hcal$. Let $t_j\to\infty$ be arbitrary. We show that a subsequence of $Y(t_j)$ converges weakly to zero.

Fix $T>0$ and, for $j$ sufficiently large, define the translated trajectories
\[
 Y_j(s)=Y(t_j+s),
 \qquad s\in[-T,T].
\]
The uniform phase bound and the equations give
\begin{align*}
 u_j,w_j,z_j&\quad\text{bounded in }L^\infty(-T,T;H_0^1(\Omega)),\\
 p_j:=u_{j,s},\ q_j:=w_{j,s}&\quad\text{bounded in }L^\infty(-T,T;L^2(\Omega)),\\
 p_{j,s}=\Delta u_j-\alpha w_j+f(u_j),\quad
 q_{j,s}=\Delta w_j+D\Delta z_j-\alpha u_j
 &\quad\text{bounded in }L^\infty(-T,T;H^{-1}(\Omega)),\\
 z_{j,s}=\tau^{-1}(q_j-z_j)&\quad\text{bounded in }L^\infty(-T,T;L^2(\Omega)).
\end{align*}
Here the bound for $p_{j,s}$ follows from \textup{(F1)} and the uniform $H_0^1$-bound. The compactness of the translated configuration variables follows from the compact embedding $H_0^1(\Omega)\Subset L^r(\Omega)$, $2\le r<2^*$, together with time equicontinuity. Indeed, the bounds on $p_j=u_{j,s}$, $q_j=w_{j,s}$, and $z_{j,s}=\tau^{-1}(q_j-z_j)$ show that $u_j$, $w_j$, and $z_j$ are uniformly Lipschitz from $[-T,T]$ into $L^2(\Omega)$. If $n\ge3$, choose $\vartheta\in(0,1]$ so that
\[
 \frac1r=\frac{\vartheta}{2}+\frac{1-\vartheta}{2^*}.
\]
Interpolation with the uniform $H_0^1$-bounds gives
\begin{align*}
 \norm{u_j(t)-u_j(s)}_r
 &\le C\norm{u_j(t)-u_j(s)}_2^{\vartheta}
       \norm{u_j(t)-u_j(s)}_{H_0^1}^{1-\vartheta}\\
 &\le C\abs{t-s}^{\vartheta},
\end{align*}
and the same estimate holds for $w_j$ and $z_j$. If $n=1,2$, the same conclusion follows by fixing a finite $R>r$ and interpolating between $L^2(\Omega)$ and $L^R(\Omega)$. Thus these families are equicontinuous in $L^r(\Omega)$, while their values at each fixed time are relatively compact in $L^r(\Omega)$. The Banach-valued Arzel\`a--Ascoli theorem therefore yields, after extraction,
\begin{equation}\label{eq:translated-strong}
 u_j\to\bar u,
 \quad w_j\to\bar w,
 \quad z_j\to\bar z
 \quad\text{in }C([-T,T];L^r(\Omega))
\end{equation}
for every $2\le r<2^*$. We also record the weak temporal compactness needed for the velocity components. For every $\phi\in H_0^1(\Omega)$,
\[
 \abs{\ip{p_j(t)-p_j(s)}{\phi}}
 \le C\abs{t-s}\norm{\phi}_{H_0^1},
 \qquad
 \abs{\ip{q_j(t)-q_j(s)}{\phi}}
 \le C\abs{t-s}\norm{\phi}_{H_0^1},
\]
because $p_{j,s}$ and $q_{j,s}$ are uniformly bounded in $H^{-1}$. Since $H_0^1(\Omega)$ is dense in $L^2(\Omega)$ and $p_j,q_j$ are uniformly bounded in $L^2$, approximation of an arbitrary $L^2$ test function by $H_0^1$ functions yields scalar equicontinuity for every element of $L^2(\Omega)$. The weak topology on a bounded ball of the separable Hilbert space $L^2(\Omega)$ is metrizable; hence the scalar Arzel\`a--Ascoli argument gives relative compactness in $C_w([-T,T];L^2(\Omega))$.

For the configuration variables, one argues in the same way in the duality $H_0^1$--$H^{-1}$. Indeed, $L^2(\Omega)$ is dense in $H^{-1}(\Omega)$, the families $u_j,w_j,z_j$ are uniformly bounded in $H_0^1$, and their time derivatives are uniformly bounded in $L^2$. Hence their scalar pairings with every element of $H^{-1}$ are equicontinuous. After extraction,
\begin{align}
 (u_j,w_j,z_j)&\to(\bar u,\bar w,\bar z)
 &&\text{in }C_w([-T,T];H_0^1(\Omega)^3),\label{eq:translated-Cw-config}\\
 (p_j,q_j)&\to(\bar p,\bar q)
 &&\text{in }C_w([-T,T];L^2(\Omega)^2).\label{eq:translated-Cw-velocity}
\end{align}
Passing to the limit in
\[
 u_j(t)-u_j(s)=\int_s^t p_j(\xi)\,\dd\xi,
 \qquad
 w_j(t)-w_j(s)=\int_s^t q_j(\xi)\,\dd\xi
\]
shows that $\bar p=\bar u_s$ and $\bar q=\bar w_s$ in distributions. Since $2(p+1)<2^*$ under \textup{(F1)}, the pointwise estimate \eqref{eq:f-difference-pointwise} and \eqref{eq:translated-strong} imply
\begin{equation}\label{eq:translated-f-strong}
 f(u_j)\to f(\bar u)
 \quad\text{strongly in }C([-T,T];L^2(\Omega)).
\end{equation}
Thus one may pass to the limit in the translated equations.

On the other hand, \eqref{eq:semiflow-energy} gives
\[
 \int_0^\infty\norm{\nabla z(t)}_2^2\,\dd t<\infty.
\]
Hence
\begin{equation}\label{eq:translated-dissipation-zero}
 \int_{-T}^T\norm{\nabla z_j(s)}_2^2\,\dd s
 =\int_{t_j-T}^{t_j+T}\norm{\nabla z(t)}_2^2\,\dd t\longrightarrow0.
\end{equation}
It follows that $\bar z\equiv0$. Passing to the limit in $\tau z_{j,s}+z_j=q_j$ gives $\bar q=\bar w_s=0$. Therefore $\bar w$ is independent of $s$. The second equation of \eqref{eq:augmented-system} reduces to
\begin{equation}\label{eq:limit-second-stationary}
 -\Delta\bar w+\alpha\bar u=0.
\end{equation}
Since $\alpha\ne0$, $\bar u$ is independent of $s$ as well. The first equation then yields the stationary system
\begin{equation}\label{eq:stationary-limit}
 \begin{cases}
 -\Delta\bar u+\alpha\bar w=f(\bar u),\\
 -\Delta\bar w+\alpha\bar u=0.
 \end{cases}
\end{equation}
Testing by $\bar u$ and $\bar w$ and adding gives
\begin{equation}\label{eq:stationary-I-zero}
 I(\bar u,\bar w)=0.
\end{equation}

Let $E_\infty=\lim_{t\to\infty}E(t)$. By \eqref{eq:translated-strong}, $u_j(0)\to\bar u$ in $L^{p+2}(\Omega)$ and $u_j(0),w_j(0)\to\bar u,\bar w$ in $L^2(\Omega)$. Weak lower semicontinuity of the coercive quadratic form therefore gives
\begin{equation}\label{eq:J-limit-bound}
 J(\bar u,\bar w)
 \le\liminf_{j\to\infty}J(u(t_j),w(t_j))
 \le E_\infty<d.
\end{equation}
If $(\bar u,\bar w)\ne(0,0)$, then \eqref{eq:stationary-I-zero} places the pair on $\Ncal$, and $J(\bar u,\bar w)\ge d$, contradicting \eqref{eq:J-limit-bound}. Thus $\bar u=\bar w=0$. Since $\bar p=\bar u_s$ and $\bar q=\bar w_s$, one has $\bar p=\bar q=0$, and also $\bar z=0$. Evaluating \eqref{eq:translated-Cw-config}--\eqref{eq:translated-Cw-velocity} at $s=0$ shows that, along the extracted subsequence,
\[
 Y(t_j)\rightharpoonup0\qquad\text{in }\Hcal.
\]
Hence every sequence $t_j\to\infty$ has a subsequence converging weakly to zero. Since the orbit is bounded, the standard sequential criterion in a Hilbert space gives \eqref{eq:weak-convergence-H}.

Finally, compact embedding upgrades weak convergence of the configuration components to \eqref{eq:strong-lower-convergence}. The potential convergence follows from \eqref{eq:F-eps} and strong convergence in $L^{p+2}(\Omega)$; the coupling term follows from strong $L^2$ convergence.
\end{proof}

\begin{corollary}[Precompactness upgrade to strong asymptotic stability]\label{cor:strong-precompact}
Under the hypotheses of Theorem~\ref{thm:weak-asymptotic}, if the orbit
\begin{equation}\label{eq:precompact-orbit}
 \{Y(t):t\ge0\}
\end{equation}
is relatively compact in $\Hcal$, then
\begin{equation}\label{eq:strong-H-convergence}
 Y(t)\to0
 \qquad\text{strongly in }\Hcal,
\end{equation}
and
\begin{equation}\label{eq:energy-to-zero}
 E(t)\to0.
\end{equation}
\end{corollary}

\begin{proof}
If \eqref{eq:strong-H-convergence} failed, there would exist $\varepsilon>0$ and $t_j\to\infty$ such that $\norm{Y(t_j)}_{\Hcal}\ge\varepsilon$. Relative compactness gives a strongly convergent subsequence. By Theorem~\ref{thm:weak-asymptotic}, its weak limit is zero, hence its strong limit is zero, a contradiction. Continuity of the energy on $\Hcal$ gives \eqref{eq:energy-to-zero}.
\end{proof}

\subsection{High-frequency flux and finite-variation criteria}\label{subsec:hf-criteria}
Let $(\lambda_k,e_k)$ be the eigenpairs of $A_D$:
\begin{equation}\label{eq:Dirichlet-eigenpairs}
 A_De_k=\lambda_ke_k,
 \qquad 0<\lambda_1\le\lambda_2\le\cdots,
 \qquad \lambda_k\to\infty.
\end{equation}
Let $P_N$ be the $L^2$-orthogonal projection onto $\operatorname{span}\{e_1,\dots,e_N\}$ and $Q_N=I-P_N$. These projections commute with $A_D$ and are contractions on $H_0^1(\Omega)$. Define
\begin{equation}\label{eq:sigma-star}
 \sigma_*=
 \begin{cases}
 1,&n=1,2,\\[1mm]
 1-\dfrac{(n-2)p}{2},&n\ge3.
 \end{cases}
\end{equation}
By \textup{(F1)}, $\sigma_*>0$.

For a stable solution, set
\begin{align}\label{eq:high-energy}
 \Ecal_N(t)=\frac12\Bigl(&\norm{Q_Nu_t}_2^2+\norm{Q_Nw_t}_2^2
 +\tau D\norm{\nabla Q_Nz}_2^2\notag\\
 &+Q_\alpha(Q_Nu,Q_Nw)\Bigr).
\end{align}
By Lemma~\ref{lem:Q-coercive}, there exist $c,C>0$, independent of $N$ and $t$, such that
\begin{equation}\label{eq:high-energy-equivalence}
 c\norm{\mathbf Q_NY(t)}_{\Hcal}^2
 \le\Ecal_N(t)
 \le C\norm{\mathbf Q_NY(t)}_{\Hcal}^2,
\end{equation}
where $\mathbf Q_N$ acts componentwise.

\begin{lemma}[Renormalized high-frequency balance]\label{lem:high-frequency-balance}
Let $Y(t)$ be the solution from Theorem~\ref{thm:stable-global}. Define
\begin{align}
 \Pcal_N(t)&=\int_\Omega\bigl[F(u(t))-F(P_Nu(t))\bigr]\,\dd x,
 \label{eq:potential-correction}\\
 \Rcal_N(t)&=\ip{f(u(t))-f(P_Nu(t))}{P_Nu_t(t)}.
 \label{eq:high-low-flux}
\end{align}
Then, in the integrated sense, for every $0\le S\le T$,
\begin{align}\label{eq:renormalized-high-balance}
 &\Ecal_N(T)-\Pcal_N(T)+D\int_S^T\norm{\nabla Q_Nz(t)}_2^2\,\dd t\notag\\
 &\qquad=\Ecal_N(S)-\Pcal_N(S)-\int_S^T\Rcal_N(t)\,\dd t.
\end{align}
Moreover, if $M=\sup_{t\ge0}\norm{u(t)}_{H_0^1}$, then
\begin{equation}\label{eq:potential-correction-uniform}
 \delta_N(M):=\sup_{t\ge0}\abs{\Pcal_N(t)}\longrightarrow0.
\end{equation}
For every $0<\sigma<\sigma_*$, there exists $C_{M,\sigma}>0$ such that
\begin{equation}\label{eq:source-tail-rate}
 \norm{f(u)-f(P_Nu)}_2
 \le C_{M,\sigma}\lambda_{N+1}^{-\sigma/2}
\end{equation}
whenever $\norm{u}_{H_0^1}\le M$. Consequently,
\begin{equation}\label{eq:flux-pointwise-rate}
 \abs{\Rcal_N(t)}
 \le C_{M,\sigma}\lambda_{N+1}^{-\sigma/2}\norm{P_Nu_t(t)}_2.
\end{equation}
\end{lemma}

\begin{proof}
For classical solutions, applying $Q_N$ to the first two equations in \eqref{eq:augmented-system}, testing by $Q_Nu_t$ and $Q_Nw_t$, and using $w_t=z+\tau z_t$ gives
\begin{equation}\label{eq:raw-high-balance}
 \frac{\dd}{\dd t}\Ecal_N(t)+D\norm{\nabla Q_Nz(t)}_2^2
 =\ip{f(u(t))}{Q_Nu_t(t)}.
\end{equation}
On the other hand,
\begin{align*}
 \frac{\dd}{\dd t}\Pcal_N(t)
 &=\ip{f(u)}{u_t}-\ip{f(P_Nu)}{P_Nu_t}\\
 &=\ip{f(u)}{Q_Nu_t}+\ip{f(u)-f(P_Nu)}{P_Nu_t}.
\end{align*}
Subtracting gives the differential form of \eqref{eq:renormalized-high-balance}. We now pass from classical to energy solutions. Fix $N\in\N$ and $0\le S\le T<T_{\max}$, and let $Y_m$ be the classical approximants from \eqref{eq:classical-approximation}. Since $P_N$ and $Q_N$ are bounded on every component of $\Hcal$,
\[
 \Ecal_{N,m}\to\Ecal_N,
 \qquad
 \Pcal_{N,m}\to\Pcal_N
 \quad\text{in }C([0,T]).
\]
Here the second convergence also uses $u_m\to u$ in $C([0,T];H_0^1)$ and the continuity of $\Phi$. Furthermore,
\[
 Q_Nz_m\to Q_Nz\quad\text{in }C([0,T];H_0^1),
\]
so the dissipative integrals converge. Local Lipschitz continuity of the Nemytskii map $f:H_0^1\to L^2$, together with the finite-rank boundedness of $P_N$, gives
\[
 f(u_m)-f(P_Nu_m)\to f(u)-f(P_Nu)
 \quad\text{in }C([0,T];L^2),
\]
while $P_Nu_{m,t}\to P_Nu_t$ in $C([0,T];L^2)$. Consequently,
\[
 \Rcal_{N,m}\to\Rcal_N\quad\text{in }C([0,T]),
\]
and the flux integrals converge. Passing to the limit in the integrated classical identity proves \eqref{eq:renormalized-high-balance} for every energy solution.

To prove \eqref{eq:potential-correction-uniform}, the mean-value theorem and \eqref{eq:F1-growth} imply
\begin{equation}\label{eq:F-difference}
 \abs{F(a)-F(b)}
 \le C\bigl(1+\abs{a}^{p+1}+\abs{b}^{p+1}\bigr)\abs{a-b}.
\end{equation}
Since $p+2<2^*$, interpolation with the spectral estimate
\[
 \norm{Q_Nu}_2\le\lambda_{N+1}^{-1/2}\norm{\nabla u}_2
\]
gives, in every dimension,
\[
 \sup_{\norm{u}_{H_0^1}\le M}\norm{Q_Nu}_{p+2}\longrightarrow0;
\]
when $n\le2$, one first fixes $R>p+2$ and uses $H_0^1(\Omega)\hookrightarrow L^R(\Omega)$.
H\"older's inequality in \eqref{eq:F-difference} yields \eqref{eq:potential-correction-uniform}.

We prove \eqref{eq:source-tail-rate}. Suppose first that $n\ge3$. For $0<\sigma<\sigma_*$, define $q_\sigma$ by
\begin{equation}\label{eq:q-sigma-highdim}
 \frac1{q_\sigma}=\frac\sigma2+\frac{1-\sigma}{2^*}.
\end{equation}
Then
\[
 \frac12-\frac1{q_\sigma}=\frac{1-\sigma}{n}.
\]
Set $a_\sigma=n/(1-\sigma)$. The inequality $\sigma<\sigma_*$ is exactly $pa_\sigma<2^*$. Hence $H_0^1(\Omega)\hookrightarrow L^{pa_\sigma}(\Omega)$, and H\"older's inequality gives
\begin{align*}
 \norm{f(u)-f(P_Nu)}_2
 &\le C_M\bigl(\norm{Q_Nu}_2+\norm{Q_Nu}_{q_\sigma}\bigr).
\end{align*}
Interpolation between $L^2$ and $L^{2^*}$, together with the spectral estimate
\[
 \norm{Q_Nu}_2\le\lambda_{N+1}^{-1/2}\norm{\nabla u}_2,
\]
yields
\[
 \norm{Q_Nu}_{q_\sigma}
 \le C\norm{Q_Nu}_2^\sigma\norm{\nabla Q_Nu}_2^{1-\sigma}
 \le C_M\lambda_{N+1}^{-\sigma/2}.
\]

If $n=1,2$, choose $q_\sigma=2/\sigma$ and $a_\sigma=2/(1-\sigma)$. Then $1/2=1/a_\sigma+1/q_\sigma$, and $H_0^1(\Omega)$ embeds into $L^{pa_\sigma}(\Omega)$ for every finite $pa_\sigma$. The Gagliardo--Nirenberg inequality, or interpolation between $L^2$ and $L^\infty$ when $n=1$, gives
\[
 \norm{Q_Nu}_{q_\sigma}\le C_M\lambda_{N+1}^{-\sigma/2}.
\]
Thus \eqref{eq:source-tail-rate} holds in every dimension. Equation \eqref{eq:flux-pointwise-rate} is immediate from \eqref{eq:high-low-flux}.
\end{proof}

The flux condition suggested by \eqref{eq:renormalized-high-balance} is
\begin{equation*}
 \tag{HF}
 \lim_{N\to\infty}\sup_{T\ge0}
 \abs{\int_0^T\Rcal_N(t)\,\dd t}=0.
\end{equation*}
A more directly verifiable sufficient condition is, for some $0<\sigma<\sigma_*$,
\begin{equation*}
 \tag*{$(\mathrm{SV})_\sigma$}
 \lambda_{N+1}^{-\sigma/2}
 \int_0^\infty\norm{P_Nu_t(t)}_2\,\dd t\longrightarrow0.
\end{equation*}

\begin{theorem}[High-frequency flux characterization and spectral finite variation]\label{thm:HF-SV}
Assume the hypotheses of Theorem~\ref{thm:weak-asymptotic}. Then the following assertions are equivalent:
\begin{enumerate}[label=\textup{(\roman*)},leftmargin=2.2em]
\item condition \textup{(HF)} holds;
\item the orbit $\{Y(t):t\ge0\}$ is relatively compact in $\Hcal$;
\item $Y(t)\to0$ strongly in $\Hcal$ as $t\to\infty$.
\end{enumerate}
Moreover, if $\mathrm{(SV)}_\sigma$ holds for some $0<\sigma<\sigma_*$, then these equivalent assertions are satisfied.
\end{theorem}

\begin{proof}
Assume first that \textup{(HF)} holds. Integrating \eqref{eq:renormalized-high-balance} from $0$ to $T$ and discarding the nonnegative dissipation term gives
\begin{equation}\label{eq:high-tail-bound}
 \Ecal_N(T)
 \le\Ecal_N(0)+2\delta_N(M)
 +\abs{\int_0^T\Rcal_N(t)\,\dd t}.
\end{equation}
Because the initial state belongs to $\Hcal$, $\Ecal_N(0)\to0$. Lemma~\ref{lem:high-frequency-balance} gives $\delta_N(M)\to0$, and hence
\begin{equation}\label{eq:uniform-tail-vanishing}
 \lim_{N\to\infty}\sup_{t\ge0}\Ecal_N(t)=0.
\end{equation}
By \eqref{eq:high-energy-equivalence}, the componentwise high-frequency tails of the orbit vanish uniformly. The low-frequency part lies in a bounded subset of a finite-dimensional space. Hence the orbit is totally bounded, and therefore relatively compact, in $\Hcal$. Corollary~\ref{cor:strong-precompact} then gives strong convergence. Thus \textup{(i)} implies \textup{(ii)}, and \textup{(ii)} implies \textup{(iii)}.

Conversely, assume that the orbit is relatively compact. Since the componentwise operators $\mathbf Q_N$ are uniformly bounded on $\Hcal$ and converge strongly to zero, their convergence is uniform on the compact orbit closure. By \eqref{eq:high-energy-equivalence},
\begin{equation}\label{eq:precompact-uniform-tail}
 \sup_{t\ge0}\Ecal_N(t)\longrightarrow0.
\end{equation}
Rearranging \eqref{eq:renormalized-high-balance} with $S=0$ yields
\begin{align*}
 \int_0^T\Rcal_N(t)\,\dd t
 ={}&\Ecal_N(0)-\Pcal_N(0)-\Ecal_N(T)+\Pcal_N(T)\\
 &-D\int_0^T\norm{\nabla Q_Nz(t)}_2^2\,\dd t.
\end{align*}
Therefore,
\begin{align}\label{eq:precompact-implies-HF-bound}
 \sup_{T\ge0}\abs{\int_0^T\Rcal_N(t)\,\dd t}
 \le{}&\Ecal_N(0)+\sup_{t\ge0}\Ecal_N(t)+2\delta_N(M)\notag\\
 &+D\int_0^\infty\norm{\nabla Q_Nz(t)}_2^2\,\dd t.
\end{align}
The first three terms tend to zero by \eqref{eq:precompact-uniform-tail} and \eqref{eq:potential-correction-uniform}. Moreover,
\[
 \int_0^\infty\norm{\nabla z(t)}_2^2\,\dd t<\infty
\]
by \eqref{eq:semiflow-energy}, while $\nabla Q_Nz(t)\to0$ for almost every $t$ and
$\norm{\nabla Q_Nz(t)}_2\le\norm{\nabla z(t)}_2$. Dominated convergence shows that the last term in \eqref{eq:precompact-implies-HF-bound} also tends to zero. Thus \textup{(ii)} implies \textup{(i)}.

Finally, \textup{(iii)} implies \textup{(ii)} because a continuous trajectory converging in $\Hcal$ has relatively compact range: its restriction to every compact time interval is compact, while its tail is contained in an arbitrarily small ball about the limit. This proves the equivalence.

If $\mathrm{(SV)}_\sigma$ holds, \eqref{eq:flux-pointwise-rate} gives
\[
 \sup_{T\ge0}\abs{\int_0^T\Rcal_N(t)\,\dd t}
 \le C_{M,\sigma}\lambda_{N+1}^{-\sigma/2}
 \int_0^\infty\norm{P_Nu_t(t)}_2\,\dd t\longrightarrow0.
\]
Hence \textup{(HF)} holds.
\end{proof}

\begin{lemma}[Potential chain rule along energy trajectories]\label{lem:energy-potential-chain}
Let $T>0$ and suppose
\[
 u\in C([0,T];H_0^1(\Omega))\cap C^1([0,T];L^2(\Omega)).
\]
Then the map $t\mapsto\Phi(u(t))$, where $\Phi(u)=\int_\Omega F(u)\,\dd x$, belongs to $C^1([0,T])$ and
\begin{equation}\label{eq:energy-potential-chain}
 \frac{\dd}{\dd t}\Phi(u(t))=\ip{f(u(t))}{u_t(t)}.
\end{equation}
\end{lemma}

\begin{proof}
For $h\ne0$ with $t+h\in[0,T]$, the scalar mean-value formula gives
\begin{align*}
 \frac{\Phi(u(t+h))-\Phi(u(t))}{h}
 =\int_0^1\ip{f\bigl(u(t)+\rho[u(t+h)-u(t)]\bigr)}
 {\dfrac{u(t+h)-u(t)}{h}}\,\dd\rho.
\end{align*}
The difference quotient converges to $u_t(t)$ in $L^2(\Omega)$. Since $u(t+h)\to u(t)$ in $H_0^1(\Omega)$ and the Nemytskii map $f:H_0^1(\Omega)\to L^2(\Omega)$ is locally Lipschitz by Lemma~\ref{lem:source-estimates}, the first factor converges to $f(u(t))$ in $L^2(\Omega)$, uniformly for $\rho\in[0,1]$. This proves \eqref{eq:energy-potential-chain}; continuity of the derivative follows from the continuity of $f(u)$ and $u_t$ in $L^2(\Omega)$.
\end{proof}

For a stable energy solution, set $g(t)=f(u(t))$. We shall use the nonlinear-force variation condition
\begin{equation*}
 \tag{SFV}
 g\in AC_{\mathrm{loc}}([0,\infty);L^2(\Omega)),
 \qquad
 \Vcal(u):=\int_0^\infty\norm{g_t(t)}_2\,\dd t<\infty.
\end{equation*}
For a regular trajectory, the chain rule in Appendix~\ref{app:regular-estimates} identifies $g_t=f'(u)u_t$ in $L^2(\Omega)$.

\begin{corollary}[Finite variation of the nonlinear force]\label{cor:source-variation}
Assume the hypotheses of Theorem~\ref{thm:weak-asymptotic}. If \textup{(SFV)} holds, then \textup{(HF)} holds. Consequently,
\begin{equation}\label{eq:SFV-strong-convergence}
 Y(t)\to0\qquad\text{strongly in }\Hcal,
 \qquad E(t)\to0.
\end{equation}
\end{corollary}

\begin{proof}
Let $\Phi$ be as in Lemma~\ref{lem:energy-potential-chain}, and define
\[
 \mathfrak B_N(t)=\ip{g(t)}{P_Nu(t)}-\Phi(P_Nu(t)),
 \qquad
 \mathfrak B(t)=\ip{g(t)}{u(t)}-\Phi(u(t)).
\]
The Hilbert-space product rule, Lemma~\ref{lem:energy-potential-chain}, and the finite-dimensional chain rule for $P_Nu$ give, for almost every $t$,
\begin{align*}
 \mathfrak B_N'(t)
 &=\ip{g_t(t)}{P_Nu(t)}+\Rcal_N(t),\\
 \mathfrak B'(t)&=\ip{g_t(t)}{u(t)}.
\end{align*}
Hence
\begin{equation}\label{eq:SFV-flux-decomposition}
 \Rcal_N(t)=\frac{\dd}{\dd t}\bigl(\mathfrak B_N(t)-\mathfrak B(t)\bigr)
 +\ip{g_t(t)}{Q_Nu(t)}.
\end{equation}
Moreover,
\begin{equation}\label{eq:SFV-boundary-correction}
 \mathfrak B_N(t)-\mathfrak B(t)
 =-\ip{g(t)}{Q_Nu(t)}+\Pcal_N(t).
\end{equation}
Let $M=\sup_{t\ge0}\norm{u(t)}_{H_0^1}$. By Lemma~\ref{lem:source-estimates} and the stable bound,
\[
 C_M:=\sup_{t\ge0}\norm{g(t)}_2<\infty.
\]
The spectral estimate gives
\begin{equation}\label{eq:SFV-QN-L2}
 \sup_{t\ge0}\norm{Q_Nu(t)}_2
 \le M\lambda_{N+1}^{-1/2}.
\end{equation}
Integrating \eqref{eq:SFV-flux-decomposition}, using \eqref{eq:SFV-boundary-correction}, and taking the supremum in $T$ yield
\begin{align*}
 \sup_{T\ge0}\abs{\int_0^T\Rcal_N(t)\,\dd t}
 \le{}&2\delta_N(M)+2C_MM\lambda_{N+1}^{-1/2}\\
 &+M\lambda_{N+1}^{-1/2}\Vcal(u).
\end{align*}
The right-hand side tends to zero by \eqref{eq:potential-correction-uniform}. Thus \textup{(HF)} holds, and Theorem~\ref{thm:HF-SV} proves \eqref{eq:SFV-strong-convergence}.
\end{proof}

For $s\ge0$, define the fractional Dirichlet scale by
\[
 D(A_D^{s/2})
 =\left\{g=\sum_{k\ge1}g_ke_k\in L^2(\Omega):
 \sum_{k\ge1}\lambda_k^s\abs{g_k}^2<\infty\right\},
\]
endowed with its natural graph norm. Using $L^2(\Omega)$ as pivot space, set
\[
 H_{A_D}^{-s}(\Omega)=\bigl(D(A_D^{s/2})\bigr)'.
\]
Equivalently, $H_{A_D}^{-s}(\Omega)$ is the completion of $L^2(\Omega)$ with respect to the norm
\begin{equation}\label{eq:negative-dirichlet-scale}
 \norm{g}_{-s,A_D}^2
 =\sum_{k\ge1}\lambda_k^{-s}\abs{g_k}^2,
 \qquad g=\sum_{k\ge1}g_ke_k.
\end{equation}

\begin{corollary}[Finite variation in a negative spatial scale]\label{cor:negative-finite-variation}
Assume the hypotheses of Theorem~\ref{thm:weak-asymptotic}. If, for some
\begin{equation}\label{eq:s-range}
 0\le s<\sigma_*,
\end{equation}
one has
\begin{equation}\label{eq:negative-FV}
 u_t\in L^1\bigl(0,\infty;H_{A_D}^{-s}(\Omega)\bigr),
\end{equation}
then
\begin{equation}\label{eq:negative-FV-conclusion}
 Y(t)\to0
 \qquad\text{strongly in }\Hcal.
\end{equation}
In particular, $u_t\in L^1(0,\infty;L^2(\Omega))$ is sufficient.
\end{corollary}

\begin{proof}
Choose $\sigma$ with $s<\sigma<\sigma_*$. By \eqref{eq:negative-dirichlet-scale},
\begin{equation}\label{eq:PN-negative-bound}
 \norm{P_Ng}_2\le\lambda_N^{s/2}\norm{g}_{-s,A_D}.
\end{equation}
Hence
\begin{align*}
 \lambda_{N+1}^{-\sigma/2}\int_0^\infty\norm{P_Nu_t(t)}_2\,\dd t
 &\le\lambda_{N+1}^{-(\sigma-s)/2}
 \norm{u_t}_{L^1(0,\infty;H_{A_D}^{-s})}\longrightarrow0.
\end{align*}
Thus $\mathrm{(SV)}_\sigma$ holds, and Theorem~\ref{thm:HF-SV} applies.
\end{proof}

For $\delta>0$, define
\begin{align}\label{eq:fractional-phase-space}
 \Hcal_\delta={}&D(A_D^{(1+\delta)/2})\times D(A_D^{\delta/2})
 \times D(A_D^{(1+\delta)/2})\times D(A_D^{\delta/2})
 \times D(A_D^{(1+\delta)/2}).
\end{align}
Since $A_D$ has compact resolvent,
\begin{equation}\label{eq:fractional-compact-embedding}
 \Hcal_\delta\Subset\Hcal
 \qquad\text{for every }\delta>0.
\end{equation}

\begin{corollary}[Eventual fractional regularity]\label{cor:fractional-regularity}
Assume the hypotheses of Theorem~\ref{thm:weak-asymptotic}. If there exist $T_0\ge0$ and $\delta>0$ such that
\begin{equation}\label{eq:eventual-fractional-bound}
 \sup_{t\ge T_0}\norm{Y(t)}_{\Hcal_\delta}<\infty,
\end{equation}
then $Y(t)\to0$ strongly in $\Hcal$.
\end{corollary}

\begin{proof}
By \eqref{eq:fractional-compact-embedding}, the tail orbit $\{Y(t):t\ge T_0\}$ is relatively compact in $\Hcal$. The initial segment $\{Y(t):0\le t\le T_0\}$ is compact by continuity. Apply Corollary~\ref{cor:strong-precompact}.
\end{proof}

\begin{remark}[Logical scope of the criteria]\label{rem:scope-criteria}
Theorem~\ref{thm:weak-asymptotic} is unconditional within the stable set. Theorem~\ref{thm:HF-SV} identifies \textup{(HF)} with orbit precompactness and strong convergence. For the linearized flow the nonlinear flux $\Rcal_N$ vanishes identically, so \textup{(HF)} is automatic and Proposition~\ref{prop:linear-strong-stability} gives strong stability. The conditions $\mathrm{(SV)}_\sigma$, \eqref{eq:negative-FV}, and \textup{(SFV)} are sufficient trajectory hypotheses for the nonlinear flux property, whereas \eqref{eq:eventual-fractional-bound} gives compactness directly. The first-order nonlinear energy identity supplies none of these additional time-integrability or regularity bounds; the orbitwise formulation is consistent with Proposition~\ref{prop:spectral-obstruction} and Remark~\ref{rem:no-linear-compactification}.
\end{remark}

\subsection{Regular trajectories and a source-variation upgrade}\label{subsec:regular-upgrade}
Let $V=H_0^1(\Omega)$, $H=L^2(\Omega)$, and $A_D=-\Delta$ with
$D(A_D)=H^2(\Omega)\cap H_0^1(\Omega)$. Define the regular phase space
\begin{equation}\label{eq:regular-phase-space}
 \Hcal_1=D(A_D)\times V\times D(A_D)\times V\times D(A_D),
\end{equation}
which agrees with \eqref{eq:fractional-phase-space} for $\delta=1$. For
$Y=(u,u_t,w,w_t,z)\in\Hcal_1$, set
\begin{align}\label{eq:regular-energy}
 \Escr_1(Y)={}&\frac12\norm{A_D^{1/2}u_t}_2^2
 +\frac12\norm{A_D^{1/2}w_t}_2^2
 +\frac{\tau D}{2}\norm{A_Dz}_2^2\notag\\
 &+\frac12\norm{A_Du}_2^2+\frac12\norm{A_Dw}_2^2
 +\alpha\ip{A_D^{1/2}u}{A_D^{1/2}w}.
\end{align}
The condition $\abs{\alpha}<\lambda_1$ gives constants $c_1,c_2>0$ such that
\begin{equation}\label{eq:regular-energy-equivalence}
 c_1\norm{Y}_{\Hcal_1}^2\le\Escr_1(Y)\le c_2\norm{Y}_{\Hcal_1}^2.
\end{equation}
Indeed, $\norm{A_D^{1/2}\phi}_2^2\le\lambda_1^{-1}\norm{A_D\phi}_2^2$ for $\phi\in D(A_D)$.

\begin{theorem}[Global propagation of regularity]\label{thm:regular-propagation}
Assume the hypotheses of Theorem~\ref{thm:stable-global}, and suppose in addition that $Y_0\in\Hcal_1$. Then the global energy solution is the unique global regular solution and
\begin{equation}\label{eq:regularity-class}
 Y\in C([0,\infty);\Hcal_1),
 \qquad
 u_{tt},w_{tt}\in C([0,\infty);L^2(\Omega)),
 \qquad
 z_t\in C([0,\infty);V).
\end{equation}
For every $t\ge0$,
\begin{align}\label{eq:regular-high-identity}
 \Escr_1(t)+D\int_0^t\norm{A_Dz(s)}_2^2\,\dd s
 =\Escr_1(0)+\int_0^t
 \ip{A_D^{1/2}f(u(s))}{A_D^{1/2}u_t(s)}\,\dd s.
\end{align}
If $M=\sup_{t\ge0}\norm{u(t)}_{V}$, then there exists $C_M>0$ such that
\begin{equation}\label{eq:regular-finite-time-bound}
 \Escr_1(t)\le\Escr_1(0)e^{C_Mt},
 \qquad t\ge0,
\end{equation}
and, for every $T>0$,
\begin{equation}\label{eq:regular-finite-time-dissipation}
 D\int_0^T\norm{A_Dz(t)}_2^2\,\dd t
 \le \Escr_1(0)e^{C_MT}.
\end{equation}
\end{theorem}

The proof, including the dimension-dependent regular Nemytskii estimate and the high-order Galerkin passage, is given in Appendix~\ref{app:regular-estimates}.

\begin{remark}[Propagation is not positive-time smoothing]\label{rem:regular-no-smoothing}
Estimate \eqref{eq:regular-finite-time-bound} is uniform on every compact time interval but may grow exponentially with the endpoint. It neither gives an eventual $\Hcal_1$-bound for arbitrary energy data nor a uniform bound of the regular orbit. This is consistent with Remark~\ref{rem:no-linear-compactification}: the linear flow has no positive-time compactifying effect.
\end{remark}

\begin{theorem}[Uniform regular bound under finite nonlinear-force variation]\label{thm:regular-SFV}
Assume the hypotheses of Theorem~\ref{thm:regular-propagation} and condition \textup{(SFV)}. Then
\begin{equation}\label{eq:regular-SFV-conclusion}
 \sup_{t\ge0}\norm{Y(t)}_{\Hcal_1}<\infty,
 \qquad
 \int_0^\infty\norm{A_Dz(t)}_2^2\,\dd t<\infty.
\end{equation}
More precisely, there exists $C_M>0$ such that
\begin{equation}\label{eq:regular-SFV-explicit}
 \sup_{t\ge0}\sqrt{1+\Escr_1(t)}
 \le C_M\left(\sqrt{1+\Escr_1(0)}+\Vcal(u)\right).
\end{equation}
If also $0<\abs{\alpha}<\lambda_1$, then the strong convergence conclusion \eqref{eq:SFV-strong-convergence} holds; alternatively, it follows from \eqref{eq:regular-SFV-conclusion}, the compact embedding $\Hcal_1\Subset\Hcal$, and Theorem~\ref{thm:HF-SV}.
\end{theorem}

\begin{proof}
Write $g(t)=f(u(t))$. By Theorem~\ref{thm:regular-propagation} and Lemma~\ref{lem:regular-composition-general},
\[
 g\in C([0,\infty);V)\cap C^1([0,\infty);H),
 \qquad g_t=f'(u)u_t.
\]
The stable energy bound and Lemma~\ref{lem:source-estimates} give
\begin{equation}\label{eq:regular-source-L2-bound}
 \sup_{t\ge0}\norm{g(t)}_2\le C_M.
\end{equation}
Consequently, by Young's inequality and \eqref{eq:regular-energy-equivalence}, one may choose $K_M>0$ such that
\begin{equation}\label{eq:regular-correction-bound}
 \abs{\ip{g(t)}{A_Du(t)}}
 \le\frac14\Escr_1(t)+K_M,
 \qquad t\ge0.
\end{equation}
Define the corrected regular energy
\begin{equation}\label{eq:corrected-regular-energy}
 \Lcal(t)=\Escr_1(t)-\ip{g(t)}{A_Du(t)}+K_M+1.
\end{equation}
Then there exist $c_M,C_M>0$ such that
\begin{equation}\label{eq:corrected-regular-equivalence}
 c_M(1+\Escr_1(t))\le\Lcal(t)\le C_M(1+\Escr_1(t)).
\end{equation}
The mixed product rule of Lemma~\ref{lem:regular-mixed-product} gives
\[
 \frac{\dd}{\dd t}\ip{g}{A_Du}
 =\ip{g_t}{A_Du}+\ip{A_D^{1/2}g}{A_D^{1/2}u_t}.
\]
Subtracting this identity from the differential form of \eqref{eq:regular-high-identity} yields
\begin{equation}\label{eq:corrected-regular-identity}
 \Lcal'(t)+D\norm{A_Dz(t)}_2^2
 =-\ip{g_t(t)}{A_Du(t)}.
\end{equation}
By \eqref{eq:corrected-regular-equivalence} and \eqref{eq:regular-energy-equivalence},
\begin{equation}\label{eq:corrected-sqrt-differential}
 \Lcal'(t)+D\norm{A_Dz(t)}_2^2
 \le C_M\norm{g_t(t)}_2\sqrt{\Lcal(t)}.
\end{equation}
For $\varepsilon>0$, divide by $2\sqrt{\Lcal+\varepsilon}$ and discard the nonnegative dissipation. After integration and letting $\varepsilon\downarrow0$,
\begin{equation}\label{eq:corrected-sqrt-bound}
 \sqrt{\Lcal(t)}
 \le\sqrt{\Lcal(0)}+C_M\int_0^t\norm{g_t(s)}_2\,\dd s.
\end{equation}
Condition \textup{(SFV)} and \eqref{eq:corrected-regular-equivalence} prove \eqref{eq:regular-SFV-explicit} and the first assertion in \eqref{eq:regular-SFV-conclusion}. Integrating \eqref{eq:corrected-regular-identity}, using $\Lcal(t)\ge0$, and then applying \eqref{eq:corrected-sqrt-bound}, we obtain
\begin{align*}
 D\int_0^T\norm{A_Dz(t)}_2^2\,\dd t
 &\le\Lcal(0)+\int_0^T\norm{g_t(t)}_2\norm{A_Du(t)}_2\,\dd t\\
 &\le\Lcal(0)+C_M\Vcal(u)
 \left(\sqrt{\Lcal(0)}+C_M\Vcal(u)\right).
\end{align*}
Letting $T\to\infty$ proves the second assertion in \eqref{eq:regular-SFV-conclusion}.
\end{proof}

\section{Finite-time blow-up below the well depth}\label{sec:below-blowup}
We first show that the unstable side of the potential well is invariant.

\begin{proposition}[Invariance of the unstable set]\label{prop:unstable-invariance}
Assume \eqref{eq:parameter-assumptions} and \textup{(F1)}--\textup{(F4)}. Let $Y$ be the maximal solution from Theorem~\ref{thm:local-wp}. If
\begin{equation}\label{eq:unstable-initial}
 E(0)<d,
 \qquad I(u_0,w_0)<0,
\end{equation}
then
\begin{equation}\label{eq:unstable-sign-invariance}
 I(u(t),w(t))<0,
 \qquad 0\le t<T_{\max},
\end{equation}
and
\begin{equation}\label{eq:unstable-Q-kappa}
 Q_\alpha(u(t),w(t))\ge\kappa_0,
 \qquad 0\le t<T_{\max}.
\end{equation}
\end{proposition}

\begin{proof}
The maps $t\mapsto I(u(t),w(t))$ and $t\mapsto J(u(t),w(t))$ are continuous. Suppose there is a first time $t_*>0$ at which $I(u(t_*),w(t_*))=0$. If the pair is nonzero, it belongs to $\Ncal$, while
\[
 J(u(t_*),w(t_*))\le E(t_*)\le E(0)<d,
\]
contradicting the definition of $d$. If the pair is zero, continuity and Lemma~\ref{lem:local-Nehari-positive} imply nonnegativity of $I(u(t),w(t))$ for all $t<t_*$ sufficiently close to $t_*$. This contradicts negativity before first contact. Hence \eqref{eq:unstable-sign-invariance} holds. The lower bound \eqref{eq:unstable-Q-kappa} follows from Lemma~\ref{lem:Nehari-barrier}.
\end{proof}

The concavity functional must reflect the one-sided MGT dissipation. In particular, no history term involving $\nabla u$ is introduced.

\begin{theorem}[Below-depth blow-up]\label{thm:below-depth-blowup}
Assume \eqref{eq:parameter-assumptions} and \textup{(F1)}--\textup{(F4)}. If \eqref{eq:unstable-initial} holds, then
\begin{equation}\label{eq:below-blowup-time}
 T_{\max}<\infty.
\end{equation}
Moreover,
\begin{equation}\label{eq:below-gradient-blowup}
 \limsup_{t\uparrow T_{\max}}
 \bigl(\norm{\nabla u(t)}_2^2+\norm{\nabla w(t)}_2^2\bigr)=\infty.
\end{equation}
No sign condition on $(u_1,w_1)$ is required.
\end{theorem}

\begin{proof}
Assume, to the contrary, that $T_{\max}=\infty$. Choose
\begin{equation}\label{eq:beta-choice}
 0<\beta<2\bigl(d-E(0)\bigr).
\end{equation}
Fix a terminal time $T>0$ and a parameter $t_0>0$, to be selected later. With $v=w-\tau z$, define on $[0,T]$
\begin{align}\label{eq:below-Psi}
 \Psi(t)={}&\norm{u(t)}_2^2+\norm{w(t)}_2^2
 +D\int_0^t\norm{\nabla v(s)}_2^2\,\dd s\notag\\
 &+D(T-t)\norm{\nabla v_0}_2^2+\beta(t+t_0)^2.
\end{align}
Note that $\Psi$ is strictly positive on $[0,T]$. By Remark~\ref{rem:recover-v},
\begin{equation}\label{eq:below-Psi-prime}
 \Psi'(t)=2\ip{u}{u_t}+2\ip{w}{w_t}
 +2D\int_0^t\ip{\nabla v(s)}{\nabla z(s)}\,\dd s+2\beta(t+t_0).
\end{equation}
We justify the second derivative at the energy level by approximation with data in $D(\Acal)$. The corresponding classical solutions converge in $C([0,T];\Hcal)$ by \eqref{eq:classical-approximation}; their functionals and first derivatives converge uniformly, while the right-hand sides of the classical second-derivative identities converge in $L^1(0,T)$. For the nonlinear term this follows from the continuity of $f:H_0^1(\Omega)\to L^2(\Omega)$ and of the $L^2$-pairing. Consequently,
\[
 \Psi\in W^{2,1}(0,T),
\]
and, for almost every $t\in(0,T)$,
\begin{equation}\label{eq:below-Psi-second-raw}
 \Psi''(t)=2\norm{u_t}_2^2+2\norm{w_t}_2^2-2Q_\alpha(u,w)
 +2\int_\Omega uf(u)\,\dd x-2\tau D\norm{\nabla z}_2^2+2\beta.
\end{equation}
The history cancellation is
\[
 -2D\ip{\nabla w}{\nabla z}+2D\ip{\nabla v}{\nabla z}
 =-2\tau D\norm{\nabla z}_2^2.
\]
Set $K(t)=\norm{u_t(t)}_2^2+\norm{w_t(t)}_2^2$. Since $uf(u)=\theta F(u)+\theta H_\theta(u)$ and
\[
 \int_\Omega F(u)\,\dd x
 =\frac12K(t)+\frac{\tau D}{2}\norm{\nabla z(t)}_2^2
 +\frac12Q_\alpha(u,w)-E(t),
\]
formula \eqref{eq:below-Psi-second-raw} becomes
\begin{align}\label{eq:below-Psi-second}
 \Psi''(t)={}&(\theta+2)K(t)+(\theta-2)\tau D\norm{\nabla z(t)}_2^2
 +(\theta-2)Q_\alpha(u,w)\notag\\
 &+2\theta\int_\Omega H_\theta(u)\,\dd x-2\theta E(t)+2\beta.
\end{align}
Using the exact energy identity,
\begin{align}\label{eq:below-Psi-second-energy}
 \Psi''(t)={}&(\theta+2)K(t)+(\theta-2)\tau D\norm{\nabla z(t)}_2^2
 +2\theta D\int_0^t\norm{\nabla z(s)}_2^2\,\dd s\notag\\
 &+(\theta-2)Q_\alpha(u,w)+2\theta\int_\Omega H_\theta(u)\,\dd x
 -2\theta E(0)+2\beta.
\end{align}
Define
\begin{align}
 X_1(t)&=\norm{u(t)}_2^2+\norm{w(t)}_2^2
 +D\int_0^t\norm{\nabla v(s)}_2^2\,\dd s+\beta(t+t_0)^2,
 \label{eq:below-X1}\\
 X_2(t)&=K(t)+D\int_0^t\norm{\nabla z(s)}_2^2\,\dd s+\beta.
 \label{eq:below-X2}
\end{align}
Cauchy--Schwarz applied to \eqref{eq:below-Psi-prime} gives
\begin{equation}\label{eq:below-CS}
 \bigl(\Psi'(t)\bigr)^2\le4X_1(t)X_2(t)\le4\Psi(t)X_2(t).
\end{equation}
Let
\begin{equation}\label{eq:eta-concavity}
 \eta=\frac{\theta-2}{4}>0,
 \qquad 4(1+\eta)=\theta+2.
\end{equation}
Subtracting $(\theta+2)X_2(t)$ from \eqref{eq:below-Psi-second-energy} and discarding nonnegative $z$-terms yields
\begin{align}\label{eq:below-core-lower}
 \Psi''(t)-4(1+\eta)X_2(t)
 \ge{}&(\theta-2)Q_\alpha(u,w)+2\theta\int_\Omega H_\theta(u)\,\dd x\notag\\
 &-2\theta E(0)-\theta\beta.
\end{align}
By Proposition~\ref{prop:unstable-invariance}, $I(u(t),w(t))<0$. Lemma~\ref{lem:Nehari-barrier} and \eqref{eq:beta-choice} therefore give
\begin{equation}\label{eq:below-positive-core}
 \Psi''(t)-4(1+\eta)X_2(t)
 \ge2\theta\bigl(d-E(0)\bigr)-\theta\beta>0.
\end{equation}
Combining \eqref{eq:below-CS} and \eqref{eq:below-positive-core},
\begin{equation}\label{eq:below-concavity-ineq}
 \Psi(t)\Psi''(t)-(1+\eta)\bigl(\Psi'(t)\bigr)^2\ge0
\end{equation}
for almost every $t\in(0,T)$. Since $\Psi$ is strictly positive and belongs to $W^{2,1}(0,T)$, the Sobolev chain rule gives $S=\Psi^{-\eta}\in W^{2,1}(0,T)$ and $S''\le0$ almost everywhere. Hence $S$ is concave and
\begin{equation}\label{eq:below-secant}
 S(T)\le S(0)+S'(0)T
 =\Psi(0)^{-\eta-1}\bigl(\Psi(0)-\eta\Psi'(0)T\bigr).
\end{equation}
Set
\[
 A=D\norm{\nabla v_0}_2^2,
 \qquad B_0=\norm{u_0}_2^2+\norm{w_0}_2^2+\beta t_0^2.
\]
Then
\[
 \Psi(0)=B_0+AT,
 \qquad
 \Psi'(0)=2\ip{u_0}{u_1}+2\ip{w_0}{w_1}+2\beta t_0,
\]
and
\begin{equation}\label{eq:below-terminal-choice}
 \Psi(0)-\eta\Psi'(0)T
 =B_0+T\bigl(A-\eta\Psi'(0)\bigr).
\end{equation}
First choose $t_0$ so large that $A-\eta\Psi'(0)<0$, and then choose $T$ large enough to make the right-hand side of \eqref{eq:below-terminal-choice} negative. Formula \eqref{eq:below-secant} would imply $S(T)<0$, a contradiction. Hence $T_{\max}<\infty$.

To prove \eqref{eq:below-gradient-blowup}, suppose the two spatial gradients remain bounded on $[0,T_{\max})$. The subcritical estimate \eqref{eq:F-eps} and Sobolev embedding bound $\int_\Omega F(u(t))\,\dd x$ from above. Since $E(t)\le E(0)$, the energy representation and coercivity of $Q_\alpha$ then bound every component of $Y(t)$ in $\Hcal$, contradicting \eqref{eq:continuation-alternative}.
\end{proof}

The critical level can be classified without a new concavity estimate. The argument combines strict energy loss away from the zero-dissipation set with the Nehari transversality in Lemma~\ref{lem:Nehari-transversality}; compare the critical-depth framework in \cite{XuCritical}.

\begin{theorem}[Dynamics at the critical well depth]\label{thm:critical-depth}
Assume \eqref{eq:parameter-assumptions}, \textup{(F1)}--\textup{(F4)}, and $\alpha\ne0$. Let $Y_0\in\Hcal$ satisfy
\begin{equation}\label{eq:critical-energy}
 E(0)=d.
\end{equation}
Then the following trichotomy holds.
\begin{enumerate}[label=\textup{(\roman*)},leftmargin=2.2em]
\item If $I(u_0,w_0)<0$, then $T_{\max}<\infty$ and \eqref{eq:below-gradient-blowup} holds.
\item If $I(u_0,w_0)>0$, or if $(u_0,w_0)=(0,0)$, then there exists $t_0>0$ such that
\begin{equation}\label{eq:critical-entry-stable}
 E(t_0)<d,
 \qquad (u(t_0),w(t_0))\in\Wcal.
\end{equation}
Consequently, the solution is global and satisfies the weak convergence conclusions of Theorem~\ref{thm:weak-asymptotic}.
\item If $I(u_0,w_0)=0$ and $(u_0,w_0)\ne(0,0)$, then
\begin{equation}\label{eq:critical-ground-state-data}
 J(u_0,w_0)=d,
 \qquad u_1=w_1=z_0=0,
 \qquad J'(u_0,w_0)=0.
\end{equation}
Thus $(u_0,w_0)$ is a Nehari ground state and the corresponding solution is stationary.
\end{enumerate}
\end{theorem}

\begin{proof}
We first record the structure of trajectories on which the dissipation vanishes. If $z\equiv0$ on an interval $[0,T]$ with $T>0$, then the relation $\tau z_t+z=w_t$ gives $w_t\equiv0$. The second equation of \eqref{eq:augmented-system} reduces to
\[
 A_Dw+\alpha u=0.
\]
Since $\alpha\ne0$, $u$ is independent of time; the first equation then shows that the state is stationary on $[0,T]$, and uniqueness extends it as a stationary solution on its maximal interval. In particular, every nonstationary solution satisfies
\begin{equation}\label{eq:critical-strict-energy-drop}
 E(t)=d-D\int_0^t\norm{\nabla z(s)}_2^2\,\dd s<d,
 \qquad t>0.
\end{equation}
Every stationary solution satisfies $I(u,w)=0$, as follows by testing the two stationary equations by $u$ and $w$ and adding.

Suppose first that $I(u_0,w_0)<0$. The solution cannot be stationary. By continuity, $I(u(t),w(t))<0$ for $0\le t\le t_0$ with some $t_0>0$, while \eqref{eq:critical-strict-energy-drop} gives $E(t_0)<d$. Restarting the flow at $t_0$ and applying Theorem~\ref{thm:below-depth-blowup} proves alternative \textup{(i)}.

If $I(u_0,w_0)>0$, the same argument gives a time $t_0>0$ for which $I(u(t_0),w(t_0))>0$ and $E(t_0)<d$. Since $J(u(t_0),w(t_0))\le E(t_0)<d$, the pair belongs to $\Wcal$. Theorems~\ref{thm:stable-global} and \ref{thm:weak-asymptotic}, applied after time translation, yield alternative \textup{(ii)}.

Assume next that $(u_0,w_0)=(0,0)$. Since $d>0$, the initial state is not stationary. Choose $\delta>0$ so small that
\[
 Q_\alpha(u(t),w(t))\le\rho_0,
 \qquad 0\le t\le\delta.
\]
The pair $(u(t),w(t))$ cannot vanish identically on $[0,\delta]$, for otherwise the equations would force the full state to vanish and hence $E(0)=0$. Choose $t_0\in(0,\delta]$ with $(u(t_0),w(t_0))\ne(0,0)$. Lemma~\ref{lem:local-Nehari-positive} and \eqref{eq:critical-strict-energy-drop} give
\[
 I(u(t_0),w(t_0))>0,
 \qquad E(t_0)<d,
\]
so \eqref{eq:critical-entry-stable} follows and the preceding stable argument applies.

Finally, suppose that $I(u_0,w_0)=0$ and $(u_0,w_0)\ne(0,0)$. Then $(u_0,w_0)\in\Ncal$, and therefore $J(u_0,w_0)\ge d$. Since
\[
 d=E(0)=\frac12\norm{u_1}_2^2+\frac12\norm{w_1}_2^2
 +\frac{\tau D}{2}\norm{\nabla z_0}_2^2+J(u_0,w_0),
\]
all nonnegative terms must attain equality. Hence $J(u_0,w_0)=d$ and $u_1=w_1=z_0=0$. Lemma~\ref{lem:Nehari-transversality} gives $J'(u_0,w_0)=0$, that is,
\begin{equation}\label{eq:critical-ground-state-system}
 \begin{cases}
 A_Du_0+\alpha w_0=f(u_0),\\
 A_Dw_0+\alpha u_0=0.
 \end{cases}
\end{equation}
Since $f(u_0)\in L^2(\Omega)$, elliptic regularity in \eqref{eq:critical-ground-state-system} gives $u_0,w_0\in D(A_D)$. Set $Y_*=(u_0,0,w_0,0,0)$. Then $Y_*\in D(\Acal)$ and $\Acal Y_*+\mathcal G(Y_*)=0$. Thus $Y_*$ is a stationary solution, and uniqueness gives $Y(t)\equiv Y_*$, proving alternative \textup{(iii)}.
\end{proof}

\section{Positive-energy and prescribed-level blow-up}\label{sec:positive-blowup}
The below-depth argument uses the Nehari barrier. We now derive a separate mechanism that does not involve $d$. The outward-projection argument is in the line of the positive-energy concavity methods in \cite{LevineTodorova,GazzolaSquassina}; the present adaptation is based on a shifted MGT-history functional.

\begin{lemma}[Positive-energy concavity criterion]\label{lem:positive-concavity}
Let $a>1$, $B\ge0$, $0<T_*\le\infty$, and let $\psi\in W_{\mathrm{loc}}^{2,1}([0,T_*))$ satisfy $\psi(0)>0$. Define the maximal positivity time
\begin{equation}\label{eq:positive-interval}
 T_+=\sup\bigl\{T\in(0,T_*]:\psi(t)>0\ \text{for every }0\le t<T\bigr\}.
\end{equation}
Assume that
\begin{equation}\label{eq:positive-concavity-assumption}
 \psi''(t)\ge a\frac{(\psi'(t))^2}{\psi(t)}-B
\end{equation}
for almost every $t\in(0,T_+)$, and that
\begin{equation}\label{eq:positive-concavity-initial}
 \psi'(0)>0,
 \qquad
 (\psi'(0))^2>\frac{2B}{2a-1}\psi(0).
\end{equation}
Then $T_+=T_*$, the continuation time $T_*$ is finite, and
\begin{equation}\label{eq:positive-time-bound}
 T_*\le\frac{\psi(0)^{1-a}}{(a-1)R_0^{1/2}},
 \qquad
 R_0=\psi(0)^{-2a}
 \left((\psi'(0))^2-\frac{2B}{2a-1}\psi(0)\right)>0.
\end{equation}
In particular, no positive finite continuation can exist on $[0,\infty)$.
\end{lemma}

\begin{proof}
Since $\psi\in W_{\mathrm{loc}}^{2,1}$, one has $\psi\in C^1$, and therefore $T_+>0$. Set
\[
 C=\frac{2B}{2a-1},
 \qquad
 R(t)=(\psi'(t))^2\psi(t)^{-2a}-C\psi(t)^{1-2a},
 \qquad 0\le t<T_+.
\]
Let
\[
 \tau_+=\sup\bigl\{T\in(0,T_+]:\psi'(t)>0\ \text{for every }0\le t<T\bigr\}.
\]
Condition \eqref{eq:positive-concavity-initial} and continuity of $\psi'$ imply $\tau_+>0$. On $(0,\tau_+)$, the Sobolev chain rule gives, for almost every $t$,
\[
 R'(t)=2\psi'(t)\psi(t)^{-2a}
 \left(\psi''(t)-a\frac{(\psi'(t))^2}{\psi(t)}+B\right)\ge0.
\]
Hence $R(t)\ge R(0)=R_0>0$ and, since $\psi'(t)>0$ on this interval,
\begin{equation}\label{eq:positive-derivative-lower}
 \psi'(t)\ge R_0^{1/2}\psi(t)^a,
 \qquad 0\le t<\tau_+.
\end{equation}
If $\tau_+<T_+$, continuity and \eqref{eq:positive-derivative-lower} would give
\[
 \psi'(\tau_+)\ge R_0^{1/2}\psi(\tau_+)^a>0,
\]
which extends the strict positivity of $\psi'$ beyond $\tau_+$ and contradicts its definition. Thus $\tau_+=T_+$. Consequently, $\psi$ is increasing and $\psi(t)\ge\psi(0)>0$ for every $t<T_+$. If $T_+<T_*$, continuity at $T_+$ would extend positivity beyond $T_+$, again contradicting the definition of $T_+$. Hence $T_+=T_*$.

It follows from \eqref{eq:positive-derivative-lower} that
\[
 \frac{\dd}{\dd t}\psi(t)^{1-a}
 =(1-a)\psi(t)^{-a}\psi'(t)
 \le-(a-1)R_0^{1/2}
\]
for almost every $t<T_*$. Therefore
\[
 0<\psi(t)^{1-a}
 \le\psi(0)^{1-a}-(a-1)R_0^{1/2}t,
 \qquad 0\le t<T_*.
\]
The right-hand side must remain positive for every $t<T_*$, which proves \eqref{eq:positive-time-bound}.
\end{proof}

Recall that $v_0=w_0-\tau z_0$.

\begin{theorem}[Positive-energy blow-up for arbitrary initial MGT displacement]\label{thm:positive-blowup}
Assume \eqref{eq:parameter-assumptions}, \textup{(F1)}, and \textup{(F2)}. Let $Y_0\in\Hcal$, set $v_0=w_0-\tau z_0$, and define
\begin{equation}\label{eq:M0P0}
 M_0=\norm{u_0}_2^2+\norm{w_0}_2^2,
 \qquad
 P_0=\ip{u_0}{u_1}+\ip{w_0}{w_1},
\end{equation}
\begin{equation}\label{eq:Xi0}
 \Xi_0=2E(0)+\frac{D}{\theta\tau(\theta-2)}\norm{\nabla v_0}_2^2.
\end{equation}
If
\begin{equation}\label{eq:positive-blowup-condition}
 E(0)\ge0,
 \qquad P_0>0,
 \qquad P_0^2>\Xi_0M_0,
\end{equation}
then the maximal solution blows up in finite time. More precisely,
\begin{equation}\label{eq:positive-blowup-bound}
 T_{\max}\le
 \frac{2M_0}{(\theta-2)\sqrt{P_0^2-\Xi_0M_0}}.
\end{equation}
Furthermore,
\begin{equation}\label{eq:positive-gradient-blowup}
 \limsup_{t\uparrow T_{\max}}
 \bigl(\norm{\nabla u(t)}_2^2+\norm{\nabla w(t)}_2^2\bigr)=\infty.
\end{equation}
\end{theorem}

\begin{proof}
Let $[0,T_{\max})$ be the maximal interval of existence. Since $P_0>0$, one necessarily has $M_0>0$. Use the shifted MGT history
\begin{equation}\label{eq:positive-Psi}
 \Psi(t)=\norm{u(t)}_2^2+\norm{w(t)}_2^2
 +D\int_0^t\norm{\nabla(v(s)-v_0)}_2^2\,\dd s,
 \qquad 0\le t<T_{\max}.
\end{equation}
Then $\Psi(0)=M_0>0$. Define
\begin{equation}\label{eq:Psi-positive-time}
 T_+=\sup\bigl\{T\in(0,T_{\max}]:\Psi(t)>0\ \text{for every }0\le t<T\bigr\}.
\end{equation}

Since $v_t=z$ and $v(0)=v_0$, Remark~\ref{rem:recover-v} gives
\[
 \norm{\nabla(v(t)-v_0)}_2^2
 =2\int_0^t\ip{\nabla(v(s)-v_0)}{\nabla z(s)}\,\dd s.
\]
Thus, at the classical level, differentiation of \eqref{eq:positive-Psi} gives
\begin{equation}\label{eq:positive-Psi-prime}
 \Psi'(t)=2\ip{u}{u_t}+2\ip{w}{w_t}
 +2D\int_0^t\ip{\nabla(v(s)-v_0)}{\nabla z(s)}\,\dd s.
\end{equation}

We now justify \eqref{eq:positive-Psi-prime} and the subsequent
second differentiation at the energy level. Fix $0<R<T_{\max}$ and choose $Y_{0m}\in D(\Acal)$ with $Y_{0m}\to Y_0$ in $\Hcal$. By continuous dependence, the corresponding classical solutions satisfy
\[
 Y_m\to Y\qquad\text{in }C([0,R];\Hcal)
\]
for all sufficiently large $m$. Set $v_m=w_m-\tau z_m$, $v_{0m}=w_{0m}-\tau z_{0m}$, and define $\Psi_m$ by \eqref{eq:positive-Psi} with the approximating variables. At the classical level, all identities below follow by ordinary differentiation. Strong phase convergence yields
\[
 v_m\to v\quad\text{in }C([0,R];H_0^1),
 \qquad
 \Psi_m\to\Psi\quad\text{uniformly on }[0,R],
\]
and the classical first-derivative expressions converge uniformly to the right-hand side of \eqref{eq:positive-Psi-prime}. Hence $\Psi\in C^1([0,R])$ and \eqref{eq:positive-Psi-prime} holds. The right-hand sides of the classical second-derivative identities converge in $L^1(0,R)$: the quadratic terms converge by strong phase convergence, and
\[
 \int_\Omega u_mf(u_m)\,\dd x\to\int_\Omega uf(u)\,\dd x,
 \qquad
 \int_\Omega H_\theta(u_m)\,\dd x\to\int_\Omega H_\theta(u)\,\dd x
\]
by the continuity of $f:H_0^1\to L^2$, of $\Phi:H_0^1\to\R$, and the identity $H_\theta(s)=\theta^{-1}sf(s)-F(s)$. Thus $\Psi\in W^{2,1}(0,R)$ and the second-derivative formulas below hold almost everywhere. Since $R<T_{\max}$ is arbitrary, the calculation is valid on every compact subinterval of the maximal interval.

Differentiating \eqref{eq:positive-Psi-prime} at the classical level and passing to the energy solution as above gives
\begin{align}\label{eq:positive-Psi-second-raw}
 \Psi''(t)={}&2K(t)-2Q_\alpha(u,w)+2\int_\Omega uf(u)\,\dd x
 -2\tau D\norm{\nabla z(t)}_2^2\notag\\
 &-2D\ip{\nabla v_0}{\nabla z(t)},
\end{align}
where $K(t)=\norm{u_t(t)}_2^2+\norm{w_t(t)}_2^2$. Reconstructing the source through the exact energy yields
\begin{align}\label{eq:positive-Psi-second}
 \Psi''(t)={}&(\theta+2)K(t)+(\theta-2)\tau D\norm{\nabla z(t)}_2^2
 +2\theta D\int_0^t\norm{\nabla z(s)}_2^2\,\dd s\notag\\
 &+(\theta-2)Q_\alpha(u,w)+2\theta\int_\Omega H_\theta(u)\,\dd x
 -2\theta E(0)-2D\ip{\nabla v_0}{\nabla z(t)}.
\end{align}
Young's inequality gives
\begin{equation}\label{eq:v0-young}
 -2D\ip{\nabla v_0}{\nabla z}
 \ge-(\theta-2)\tau D\norm{\nabla z}_2^2
 -\frac{D}{\tau(\theta-2)}\norm{\nabla v_0}_2^2.
\end{equation}
Set
\begin{equation}\label{eq:positive-X}
 X(t)=K(t)+D\int_0^t\norm{\nabla z(s)}_2^2\,\dd s,
\end{equation}
and
\begin{equation}\label{eq:positive-B0}
 B_0=2\theta E(0)+\frac{D}{\tau(\theta-2)}\norm{\nabla v_0}_2^2\ge0.
\end{equation}
Using $2\theta\ge\theta+2$, $Q_\alpha\ge0$, and $H_\theta\ge0$, formulas \eqref{eq:positive-Psi-second}--\eqref{eq:v0-young} imply
\begin{equation}\label{eq:positive-Psi-X}
 \Psi''(t)\ge(\theta+2)X(t)-B_0.
\end{equation}
Cauchy--Schwarz applied to \eqref{eq:positive-Psi-prime} gives
\begin{equation}\label{eq:positive-CS}
 (\Psi'(t))^2\le4\Psi(t)X(t).
\end{equation}
Consequently, on the positivity interval $(0,T_+)$,
\begin{equation}\label{eq:positive-concavity-final}
 \Psi''(t)\ge a\frac{(\Psi'(t))^2}{\Psi(t)}-B_0,
 \qquad a=\frac{\theta+2}{4}>1,
\end{equation}
for almost every $t$. At $t=0$, $\Psi(0)=M_0$ and $\Psi'(0)=2P_0$. Since $2a-1=\theta/2$ and $B_0/\theta=\Xi_0$, condition \eqref{eq:positive-blowup-condition} is exactly \eqref{eq:positive-concavity-initial}. Lemma~\ref{lem:positive-concavity}, applied with $T_*=T_{\max}$, first shows that $T_+=T_{\max}$ and then yields
\[
 T_{\max}\le\frac{\Psi(0)^{1-a}}{(a-1)R_0^{1/2}}.
\]
Substituting $\Psi(0)=M_0$, $\Psi'(0)=2P_0$, and $B_0/\theta=\Xi_0$ gives exactly \eqref{eq:positive-blowup-bound}; in particular, $T_{\max}<\infty$. The proof of \eqref{eq:positive-gradient-blowup} is identical to the continuation argument at the end of Theorem~\ref{thm:below-depth-blowup}.
\end{proof}

\begin{corollary}[Prescribed nonnegative energy]\label{cor:prescribed-energy}
Assume \eqref{eq:parameter-assumptions} and \textup{(F1)}, \textup{(F2)}, \textup{(F4)}. For every $L\ge0$, there exist initial data satisfying the hypotheses of Theorem~\ref{thm:positive-blowup} and
\begin{equation}\label{eq:prescribed-energy}
 E(0)=L.
\end{equation}
Consequently, for every $L\ge0$, the energy shell
\[
 \bigl\{Y_0\in\Hcal:E(Y_0)=L\bigr\}
\]
contains initial data whose corresponding solution blows up in finite time.
\end{corollary}

\begin{proof}
Choose $\phi\in H_0^1(\Omega)$ with
$A_F:=\int_\Omega F(\phi)\,\dd x>0$. By Lemma~\ref{lem:potential-scaling},
\[
 J(\lambda\phi,0)
 \le\frac{\lambda^2}{2}\norm{\nabla\phi}_2^2-\lambda^\theta A_F,
 \qquad \lambda\ge1.
\]
Choose $\lambda$ so large that $J(u_0,0)<0$, where $u_0=\lambda\phi$. Set
\[
 v_0=v_1=v_2=0,
 \qquad w_0=w_1=z_0=0,
\]
and define
\[
 u_1=\mu u_0,
 \qquad
 \mu=\left(\frac{2(L-J(u_0,0))}{\norm{u_0}_2^2}\right)^{1/2}>0.
\]
Then
\[
 E(0)=\frac12\mu^2\norm{u_0}_2^2+J(u_0,0)=L,
\]
while
\[
 M_0=\norm{u_0}_2^2,
 \qquad
 P_0=\mu\norm{u_0}_2^2>0.
\]
Moreover, $v_0=0$ gives $\Xi_0=2E(0)=2L$, and hence
\[
 P_0^2-\Xi_0M_0
 =P_0^2-2E(0)M_0
 =-2J(u_0,0)\norm{u_0}_2^2>0.
\]
Thus all hypotheses of Theorem~\ref{thm:positive-blowup} are satisfied, and the corresponding maximal solution blows up in finite time.
\end{proof}

\begin{remark}[Why the shifted history is preferable]
When $v_0=0$, Theorem~\ref{thm:positive-blowup} reduces to $P_0^2>2E(0)M_0$. For general $v_0$, a terminal correction of the type used in \eqref{eq:below-Psi} would make the threshold depend on an auxiliary terminal parameter. The shifted history \eqref{eq:positive-Psi} avoids this implicit compatibility issue and yields the data-only condition \eqref{eq:positive-blowup-condition}.
\end{remark}

\section{Power and logarithmic sources}\label{sec:examples}
We verify the general hypotheses for the two principal examples. To avoid repeating the stable-side alternatives, for $\rho>0$ we write $\mathrm{(C)}_\rho$ when at least one of the following holds: \textup{(HF)}; $\mathrm{(SV)}_\sigma$ for some $0<\sigma<\rho$; \eqref{eq:negative-FV} for some $0\le s<\rho$; \textup{(SFV)}; or \eqref{eq:eventual-fractional-bound} for some $\delta>0$.

\subsection{Pure-power source}
Let
\begin{equation}\label{eq:power-source}
 f(s)=\abs{s}^q s,
 \qquad q>0.
\end{equation}
Then
\begin{equation}\label{eq:power-F-H}
 F(s)=\frac1{q+2}\abs{s}^{q+2},
 \qquad \theta=q+2,
 \qquad H_{q+2}(s)\equiv0.
\end{equation}
Assumption \textup{(F1)} holds provided
\begin{equation}\label{eq:q-range}
 0<q<p^*.
\end{equation}
The focusing condition is automatic for every nonzero $\phi\in H_0^1(\Omega)$. For this source,
\begin{equation}\label{eq:power-JI}
 J_q(u,w)=\frac12Q_\alpha(u,w)-\frac1{q+2}\norm{u}_{q+2}^{q+2},
 \qquad
 I_q(u,w)=Q_\alpha(u,w)-\norm{u}_{q+2}^{q+2}.
\end{equation}
Define
\begin{equation}\label{eq:Kalphaq}
 K_{\alpha,q}=
 \sup_{(u,w)\ne(0,0)}
 \frac{\norm{u}_{q+2}}{Q_\alpha(u,w)^{1/2}}.
\end{equation}
Sobolev embedding and coercivity give $0<K_{\alpha,q}<\infty$.

\begin{proposition}[Explicit power well depth]\label{prop:power-depth}
For \eqref{eq:power-source}--\eqref{eq:q-range},
\begin{equation}\label{eq:power-depth}
 d_q=\frac{q}{2(q+2)}K_{\alpha,q}^{-2(q+2)/q}.
\end{equation}
\end{proposition}

\begin{proof}
On the power Nehari manifold, $Q_\alpha(u,w)=\norm{u}_{q+2}^{q+2}$ and
\[
 J_q(u,w)=\frac{q}{2(q+2)}Q_\alpha(u,w).
\]
The definition of $K_{\alpha,q}$ gives
\[
 Q_\alpha(u,w)=\norm{u}_{q+2}^{q+2}
 \le K_{\alpha,q}^{q+2}Q_\alpha(u,w)^{(q+2)/2},
\]
so $Q_\alpha(u,w)\ge K_{\alpha,q}^{-2(q+2)/q}$. A maximizing sequence in \eqref{eq:Kalphaq}, rescaled onto the Nehari manifold, shows sharpness.
\end{proof}

Set
\begin{equation}\label{eq:power-sigma-star}
 \sigma_{*,q}=
 \begin{cases}
 1,&n=1,2,\\[1mm]
 1-\dfrac{(n-2)q}{2},&n\ge3.
 \end{cases}
\end{equation}

\begin{corollary}[Power-source dynamics]\label{cor:power-dynamics}
Assume \eqref{eq:parameter-assumptions} and \eqref{eq:q-range}. Then:
\begin{enumerate}[label=\textup{(\roman*)},leftmargin=2.2em]
\item Every finite-energy initial state generates a unique maximal energy solution satisfying the exact energy identity and continuation alternative.
\item If $E(0)<d_q$ and $I_q(u_0,w_0)>0$, the solution is global and the stable set is invariant. If $0<\abs{\alpha}<\lambda_1$, then $Y(t)\rightharpoonup0$ in $\Hcal$, and $Y(t)\to0$ strongly under $\mathrm{(C)}_{\sigma_{*,q}}$. If, in addition, $Y_0\in\Hcal_1$, then regularity propagates as in Theorem~\ref{thm:regular-propagation}; under \textup{(SFV)}, the uniform regular bound and higher-order dissipation in \eqref{eq:regular-SFV-conclusion} hold.
\item If $0<\abs{\alpha}<\lambda_1$ and $E(0)=d_q$, the critical-level alternatives of Theorem~\ref{thm:critical-depth} hold with $I$ and $J$ replaced by $I_q$ and $J_q$.
\item If $E(0)<d_q$ and $I_q(u_0,w_0)<0$, then the maximal solution blows up in finite time.
\item If $E(0)\ge0$, $P_0>0$, and
\[
 P_0^2>
 \left(2E(0)+\frac{D}{(q+2)\tau q}\norm{\nabla v_0}_2^2\right)M_0,
\]
then finite-time blow-up occurs. Blow-up data exist at every prescribed level $L\ge0$.
\end{enumerate}
\end{corollary}

\begin{proof}
All source assumptions hold with $p=q$, $\theta=q+2$, and $H_{q+2}\equiv0$. The conclusions follow from the general results of Sections~\ref{sec:local}--\ref{sec:positive-blowup}, with $\sigma_*=\sigma_{*,q}$.
\end{proof}

\begin{remark}[The $L^2$-subcritical and energy-subcritical power ranges]\label{rem:power-range}
For $n\ge3$, the restriction $q<p^*=2/(n-2)$ in \eqref{eq:q-range} is stronger than the variational energy-subcritical condition
\[
 q+2<2^*,
 \qquad\text{equivalently}\qquad
 q<\frac{4}{n-2}.
\]
The latter condition is sufficient for the potential term $\norm{u}_{q+2}^{q+2}$, the functionals $J_q$ and $I_q$, and the variational formula \eqref{eq:power-depth} to be well-defined on $H_0^1(\Omega)$. The narrower range used in Corollary~\ref{cor:power-dynamics} is dictated instead by the strict $L^2$-subcritical framework \textup{(F1)}:
\[
 2(q+1)<2^*.
\]
This inequality ensures that $f:H_0^1(\Omega)\to L^2(\Omega)$ is locally Lipschitz and, equally importantly, leaves the positive compactness margin $\sigma_{*,q}>0$ used in the nonlinear convergence and high-frequency tail estimates. At the borderline $q=2/(n-2)$, the power source is still $L^2$-valued, but this margin collapses; for $2/(n-2)<q<4/(n-2)$, the source is in general only $H^{-1}$-valued at the energy level. Consequently, the semigroup, uniqueness, source-variation, weak-asymptotic, and high-frequency arguments of the present paper do not extend to the full variational range by a mere change of exponent. Such an extension would require a different local well-posedness and uniqueness theory, together with a reformulation of the compactness and flux estimates.
\end{remark}

\subsection{Logarithmic source}
Let
\begin{equation}\label{eq:log-source}
 f(s)=\abs{s}^{\gamma-2}s\log\abs{s},
 \qquad f(0)=0,
 \qquad \gamma>2.
\end{equation}
Its primitive is
\begin{equation}\label{eq:log-F}
 F(s)=\frac1\gamma\abs{s}^\gamma\log\abs{s}-\frac1{\gamma^2}\abs{s}^\gamma,
 \qquad F(0)=0.
\end{equation}
With $\theta=\gamma$,
\begin{equation}\label{eq:log-H}
 H_\gamma(s)=\frac1\gamma sf(s)-F(s)=\frac1{\gamma^2}\abs{s}^\gamma.
\end{equation}
Thus $H_\gamma\ge0$ and $H_\gamma(\lambda s)=\lambda^\gamma H_\gamma(s)\le H_\gamma(s)$ for $0\le\lambda\le1$. Moreover, for every $\mu>0$,
\begin{equation}\label{eq:log-derivative-growth}
 \abs{f'(s)}\le C_\mu\bigl(1+\abs{s}^{\gamma-2+\mu}\bigr).
\end{equation}
Consequently, \textup{(F1)} holds whenever
\begin{equation}\label{eq:gamma-range}
 \begin{cases}
 \gamma>2,&n=1,2,\\[1mm]
 2<\gamma<\dfrac{2(n-1)}{n-2},&n\ge3.
 \end{cases}
\end{equation}
Indeed, for $n\ge3$, choose $\mu>0$ such that $\gamma-2+\mu<2/(n-2)$. The focusing condition follows by scaling any nonzero nonnegative $\phi\in C_c^\infty(\Omega)$.

Define
\begin{align}\label{eq:log-JI}
 J_\gamma(u,w)&=\frac12Q_\alpha(u,w)-\frac1\gamma\int_\Omega\abs{u}^\gamma\log\abs{u}\,\dd x
 +\frac1{\gamma^2}\norm{u}_\gamma^\gamma,\notag\\
 I_\gamma(u,w)&=Q_\alpha(u,w)-\int_\Omega\abs{u}^\gamma\log\abs{u}\,\dd x,
\end{align}
and let $d_\gamma$ be the corresponding well depth. Set
\begin{equation}\label{eq:log-sigma-star}
 \sigma_{*,\gamma}=
 \begin{cases}
 1,&n=1,2,\\[1mm]
 1-\dfrac{(n-2)(\gamma-2)}2,&n\ge3.
 \end{cases}
\end{equation}

\begin{corollary}[Logarithmic-source dynamics]\label{cor:log-dynamics}
Assume \eqref{eq:parameter-assumptions} and \eqref{eq:gamma-range}. Then:
\begin{enumerate}[label=\textup{(\roman*)},leftmargin=2.2em]
\item Every finite-energy initial state generates a unique maximal energy solution satisfying the exact energy identity and continuation alternative.
\item If $E(0)<d_\gamma$ and $I_\gamma(u_0,w_0)>0$, the solution is global and the stable set is invariant. If $0<\abs{\alpha}<\lambda_1$, then $Y(t)\rightharpoonup0$ in $\Hcal$, and $Y(t)\to0$ strongly under $\mathrm{(C)}_{\sigma_{*,\gamma}}$. If, in addition, $Y_0\in\Hcal_1$, then regularity propagates as in Theorem~\ref{thm:regular-propagation}; under \textup{(SFV)}, the uniform regular bound and higher-order dissipation in \eqref{eq:regular-SFV-conclusion} hold.
\item If $0<\abs{\alpha}<\lambda_1$ and $E(0)=d_\gamma$, the critical-level alternatives of Theorem~\ref{thm:critical-depth} hold with $I$ and $J$ replaced by $I_\gamma$ and $J_\gamma$.
\item If $E(0)<d_\gamma$ and $I_\gamma(u_0,w_0)<0$, then the maximal solution blows up in finite time.
\item If $E(0)\ge0$, $P_0>0$, and
\[
 P_0^2>
 \left(2E(0)+\frac{D}{\gamma\tau(\gamma-2)}\norm{\nabla v_0}_2^2\right)M_0,
\]
then finite-time blow-up occurs. Blow-up data exist at every prescribed level $L\ge0$.
\end{enumerate}
\end{corollary}

\begin{proof}
The preceding verification gives \textup{(F1)}--\textup{(F4)} with $\theta=\gamma$. Given a selected $\sigma$ or $s$ below $\sigma_{*,\gamma}$, choose $\mu>0$ so small that the selected exponent remains below the threshold \eqref{eq:sigma-star} corresponding to $p=\gamma-2+\mu$. The remaining alternatives in $\mathrm{(C)}_{\sigma_{*,\gamma}}$ require no such choice. The conclusions then follow from the general results of Sections~\ref{sec:local}--\ref{sec:positive-blowup}.
\end{proof}

\begin{remark}[Logarithmic exponent range]\label{rem:log-specialization}
On the common stable branch, Corollary~\ref{cor:log-dynamics} agrees with the logarithmic theory of \cite{HaLogMGT}. The restriction \eqref{eq:gamma-range} is the $L^2$-subcritical range required by the energy-space uniqueness and nonlinear-flux arguments. Although the logarithmic potential is well-defined for $2<\gamma<2^*$ when $n\ge3$, extension of the full dynamical theory to that variational range requires a different local theory and new difference and high-frequency estimates.
\end{remark}

\begin{remark}[Further admissible sources]
The framework also covers positive finite sums of powers. If
\[
 f(s)=\sum_{j=1}^m a_j\abs{s}^{q_j}s,
 \qquad a_j>0,
 \qquad 0<q_1\le\cdots\le q_m<p^*,
\]
one may take $\theta=q_1+2$, and
\[
 H_\theta(s)=\sum_{j=1}^m a_j
 \left(\frac1\theta-\frac1{q_j+2}\right)\abs{s}^{q_j+2},
\]
which is nonnegative and radially nondecreasing.
\end{remark}

\section{Concluding remarks}
The augmented variable $w=v+\tau v_t$ identifies the exact conservative core of the wave--MGT interaction and leads simultaneously to the correct energy and the correct static potential-well variables. Within this formulation, the source remainder
\[
 H_\theta(s)=\frac1\theta sf(s)-F(s)
\]
encodes the stable--unstable dichotomy: its nonnegativity gives stable coercivity, while its radial monotonicity gives both the Nehari barrier required by the MGT-history concavity functional and the transversality needed at the critical level. This yields, in one source-independent argument, global existence in the stable well, finite-time blow-up below the well depth on the unstable side, a complete classification at $E(0)=d$ for nonzero coupling, a positive-energy criterion valid for arbitrary initial MGT displacement, and prescribed-energy blow-up beyond the Nehari level.

At the linear level, nonzero coupling gives strong stability, although the wave-branch expansion rules out a uniform exponential rate and positive-time compactification. The nonlinear stable analysis distinguishes weak and strong asymptotics. For nonzero coupling and without any compactness hypothesis, every stable trajectory converges weakly to zero, and the configuration variables converge strongly in every subcritical Lebesgue space. The renormalized high-frequency identity yields more than a sufficient criterion: within the stable class, condition \textup{(HF)} is equivalent to relative compactness of the orbit and to strong convergence in the natural energy topology. The spectral condition $\mathrm{(SV)}_\sigma$, finite variation of $u$ in a negative Dirichlet scale, and finite total variation of the nonlinear force in $L^2(\Omega)$ are explicit sufficient conditions for this exact flux property. Separately, any eventual uniform bound in a phase space with positive fractional spatial regularity gives compactness directly.

For initial states in the regular phase space, the energy solution remains regular on every finite time interval, but the direct high-order estimate may grow exponentially and does not by itself compactify the full orbit. Under finite nonlinear-force variation, the corrected energy $\Escr_1-(f(u),A_Du)$ yields the stronger conclusions
\[
 \sup_{t\ge0}\norm{Y(t)}_{\Hcal_1}<\infty,
 \qquad
 A_Dz\in L^2(0,\infty;L^2(\Omega)).
\]
It remains open whether every stable trajectory converges strongly without an additional orbit condition and whether regular stable data remain uniformly bounded in the regular phase space.

\appendix
\section{Regular-phase estimates}\label{app:regular-estimates}
This appendix supplies the regular Nemytskii, product-rule, and Galerkin arguments used in Subsection~\ref{subsec:regular-upgrade}. Throughout, $H=L^2(\Omega)$, $V=H_0^1(\Omega)$, and $A_D=-\Delta$ with Dirichlet boundary conditions.

\begin{lemma}[Regular composition and temporal chain rule]\label{lem:regular-composition-general}
Assume \textup{(F1)}. For every $R>0$ there exists $C_R>0$ such that
\begin{equation}\label{eq:appendix-regular-composition}
 f(u)\in V,
 \qquad
 \norm{A_D^{1/2}f(u)}_2\le C_R\norm{A_Du}_2
\end{equation}
whenever $u\in D(A_D)$ and $\norm{u}_V\le R$. Moreover, the map
\[
 f:D(A_D)\longrightarrow V
\]
is continuous. If $T>0$ and
\[
 u\in C([0,T];D(A_D))\cap C^1([0,T];V),
\]
then
\begin{equation}\label{eq:appendix-temporal-chain-class}
 f(u)\in C([0,T];V)\cap C^1([0,T];H),
 \qquad
 \partial_tf(u)=f'(u)u_t\quad\text{in }C([0,T];H).
\end{equation}
\end{lemma}

\begin{proof}
Suppose first that $n\ge2$. Choose $a>\max\{n,1/p\}$ so that $ap<2^*$ when $n\ge3$; such a choice is possible because $p<2/(n-2)$. When $n=2$, no upper restriction is needed. Define $b$ by
\[
 \frac12=\frac1a+\frac1b.
\]
Then $2<b<2^*$ if $n\ge3$, while $b<\infty$ if $n=2$. Assumption \textup{(F1)} and Sobolev embedding give, on every $V$-ball of radius $R$,
\begin{equation}\label{eq:appendix-fprime-a}
 \norm{f'(u)}_a
 \le C\bigl(1+\norm{u}_{ap}^p\bigr)\le C_R.
\end{equation}
Elliptic regularity and Sobolev embedding yield
\begin{equation}\label{eq:appendix-elliptic-W1b}
 \norm{\nabla u}_b\le C\norm{u}_{H^2}\le C\norm{A_Du}_2,
 \qquad u\in D(A_D).
\end{equation}
The Sobolev chain rule, H\"older's inequality, \eqref{eq:appendix-fprime-a}, and \eqref{eq:appendix-elliptic-W1b} imply
\[
 \norm{\nabla f(u)}_2
 =\norm{f'(u)\nabla u}_2
 \le C_R\norm{A_Du}_2.
\]
Since $u$ has zero trace and $f(0)=0$, the trace of $f(u)$ vanishes. This proves \eqref{eq:appendix-regular-composition} for $n\ge2$.

If $n=1$, the embeddings $H_0^1(\Omega)\hookrightarrow L^\infty(\Omega)$ and $D(A_D)\hookrightarrow W^{1,\infty}(\Omega)$ give
\[
 \norm{f'(u)}_\infty\le C_R,
 \qquad
 \norm{\nabla f(u)}_2\le C_R\norm{\nabla u}_2
 \le C_R\norm{A_Du}_2,
\]
which proves the same estimate.

We prove continuity. Let $u_m\to u$ in $D(A_D)$. For $n\ge2$, $u_m\to u$ in measure and in every Lebesgue space needed above. Since $f'$ is continuous, $f'(u_m)\to f'(u)$ in measure. Choose $\eta>0$ so small that $ap(1+\eta)<2^*$ when $n\ge3$; in dimension two choose any finite exponent. The growth bound on $f'$ shows that $\{|f'(u_m)|^a\}$ is uniformly integrable. Hence Vitali's theorem gives
\[
 f'(u_m)\to f'(u)\qquad\text{in }L^a(\Omega).
\]
Together with $\nabla u_m\to\nabla u$ in $L^b(\Omega)$, this yields
\begin{align*}
 \norm{\nabla f(u_m)-\nabla f(u)}_2
 &\le\norm{f'(u_m)-f'(u)}_a\norm{\nabla u_m}_b\\
 &\quad+\norm{f'(u)}_a\norm{\nabla u_m-\nabla u}_b\longrightarrow0.
\end{align*}
For $n=1$, convergence in $D(A_D)$ implies uniform convergence of $u_m$ and convergence of their first derivatives in $L^2$, so the same conclusion follows from continuity of $f'$. Poincar\'e's inequality proves continuity into $V$.

Finally, Lemma~\ref{lem:nemytskii-C1} states that $f:V\to H$ is continuously Fr\'echet differentiable with derivative $Df(u)h=f'(u)h$. The Banach-space chain rule therefore gives $f(u)\in C^1([0,T];H)$ and the derivative in \eqref{eq:appendix-temporal-chain-class}; continuity into $V$ follows from the first part of the lemma.
\end{proof}

\begin{lemma}[Mixed regular product rule]\label{lem:regular-mixed-product}
Under the trajectory assumptions in Lemma~\ref{lem:regular-composition-general}, the map
\[
 t\longmapsto\ip{f(u(t))}{A_Du(t)}
\]
belongs to $C^1([0,T])$ and
\begin{equation}\label{eq:appendix-mixed-product}
 \frac{\dd}{\dd t}\ip{f(u)}{A_Du}
 =\ip{f'(u)u_t}{A_Du}
 +\ip{A_D^{1/2}f(u)}{A_D^{1/2}u_t}.
\end{equation}
\end{lemma}

\begin{proof}
Set $a(t)=f(u(t))$ and $b(t)=u(t)$. Lemma~\ref{lem:regular-composition-general} gives
\[
 a\in C([0,T];V)\cap C^1([0,T];H),
 \qquad
 b\in C([0,T];D(A_D))\cap C^1([0,T];V).
\]
Extend $a$ and $b$ slightly beyond $[0,T]$ and mollify in time. For the smooth mollifications $a_\varepsilon,b_\varepsilon$,
\[
 \frac{\dd}{\dd t}\ip{a_\varepsilon}{A_Db_\varepsilon}
 =\ip{(a_\varepsilon)_t}{A_Db_\varepsilon}
 +\ip{A_D^{1/2}a_\varepsilon}{A_D^{1/2}(b_\varepsilon)_t}.
\]
The first term converges uniformly on compact subintervals because $a_t,A_Db\in C(H)$, and the second because $a,b_t\in C(V)$. Passage to the limit and the identity $a_t=f'(u)u_t$ prove \eqref{eq:appendix-mixed-product}. Endpoint values follow by continuity.
\end{proof}

\begin{proposition}[Forced linear regularity]\label{prop:forced-linear-regular}
Let $T>0$, $G\in L^1(0,T;V)$, and $Y_0\in\Hcal_1$. The forced linear system
\begin{equation}\label{eq:appendix-forced-linear}
\begin{cases}
 u_{tt}+A_Du+\alpha w=G,\\
 w_{tt}+A_Dw+DA_Dz+\alpha u=0,\\
 \tau z_t+z=w_t
\end{cases}
\end{equation}
has a unique energy solution, and this solution satisfies
\begin{equation}\label{eq:appendix-forced-regularity}
 Y\in C([0,T];\Hcal_1),
 \qquad
 u_{tt}\in L^1(0,T;H),
 \qquad
 w_{tt}\in C([0,T];H),
 \qquad
 z_t\in C([0,T];V).
\end{equation}
Moreover,
\begin{align}\label{eq:appendix-forced-identity}
 \Escr_1(t)+D\int_0^t\norm{A_Dz(s)}_2^2\,\dd s
 =\Escr_1(0)+\int_0^t
 \ip{A_D^{1/2}G(s)}{A_D^{1/2}u_t(s)}\,\dd s.
\end{align}
If $G\in C([0,T];V)$, then $u_{tt}\in C([0,T];H)$.
\end{proposition}

\begin{proof}
Let $\{e_k\}$ be the Dirichlet eigenbasis and let $P_m$ denote projection onto the first $m$ modes. Solve \eqref{eq:appendix-forced-linear} in the corresponding finite-dimensional space with initial state $\mathbf P_mY_0$ and forcing $P_mG$. Testing the first equation by $A_Du_{m,t}$ and the second by $A_Dw_{m,t}$ gives
\begin{equation}\label{eq:appendix-linear-Galerkin}
 \frac{\dd}{\dd t}\Escr_{1,m}(t)+D\norm{A_Dz_m(t)}_2^2
 =\ip{A_D^{1/2}P_mG(t)}{A_D^{1/2}u_{m,t}(t)}.
\end{equation}
Indeed, $w_{m,t}=z_m+\tau z_{m,t}$ implies
\[
 D\ip{A_Dz_m}{A_Dw_{m,t}}
 =D\norm{A_Dz_m}_2^2
 +\frac{\tau D}{2}\frac{\dd}{\dd t}\norm{A_Dz_m}_2^2,
\]
and the two coupling terms form the derivative of
$\alpha\ip{A_D^{1/2}u_m}{A_D^{1/2}w_m}$.

By \eqref{eq:regular-energy-equivalence}, for every $\varepsilon>0$,
\[
 \frac{\dd}{\dd t}\sqrt{\Escr_{1,m}(t)+\varepsilon}
 \le C\norm{G(t)}_V.
\]
Thus
\begin{equation}\label{eq:appendix-linear-apriori}
 \sup_{0\le t\le T}\sqrt{\Escr_{1,m}(t)}
 \le\sqrt{\Escr_{1,m}(0)}+C\norm{G}_{L^1(0,T;V)}.
\end{equation}
The same estimate for the difference of two projected systems shows that the Galerkin sequence is Cauchy in $C([0,T];\Hcal_1)$, since $\mathbf P_mY_0\to Y_0$ in $\Hcal_1$ and $P_mG\to G$ in $L^1(0,T;V)$. Passing to the limit proves \eqref{eq:appendix-forced-identity}. The energy solution is unique by the linear semigroup realization in Proposition~\ref{prop:linear-semigroup}, and therefore coincides with this regular limit.

The third equation gives $z_t=\tau^{-1}(w_t-z)\in C([0,T];V)$. Since $w,z\in C([0,T];D(A_D))$, the second equation yields $w_{tt}\in C([0,T];H)$. The first gives $u_{tt}\in L^1(0,T;H)$ and gives continuity when $G\in C([0,T];V)$.
\end{proof}

\begin{proof}[Proof of Theorem~\ref{thm:regular-propagation}]
Let $\mathbf P_m$ act componentwise and set $Y_{0m}=\mathbf P_mY_0$. Since $Y_{0m}\to Y_0$ in $\Hcal_1$, the initial energies converge. If $(u_0,w_0)\ne(0,0)$, continuity of $I$ gives $I(P_mu_0,P_mw_0)>0$ for all sufficiently large $m$; if $(u_0,w_0)=(0,0)$, its projections also vanish. Thus, for all large $m$,
\[
 E(Y_{0m})<d,
 \qquad
 (P_mu_0,P_mw_0)\in\Wcal.
\]
The finite-dimensional exact energy identity and the first-contact argument from Theorem~\ref{thm:stable-global} give global stable Galerkin trajectories and a constant $M>0$, independent of $m$ and $t$, such that
\begin{equation}\label{eq:appendix-lower-Galerkin-bound}
 \norm{u_m(t)}_V+\norm{w_m(t)}_V+\norm{z_m(t)}_V
 +\norm{u_{m,t}(t)}_2+\norm{w_{m,t}(t)}_2\le M.
\end{equation}
Testing at the regular level as in \eqref{eq:appendix-linear-Galerkin} gives
\begin{equation}\label{eq:appendix-nonlinear-high-identity}
 \frac{\dd}{\dd t}\Escr_{1,m}(t)+D\norm{A_Dz_m(t)}_2^2
 =\ip{A_D^{1/2}f(u_m(t))}{A_D^{1/2}u_{m,t}(t)}.
\end{equation}
Lemma~\ref{lem:regular-composition-general}, \eqref{eq:appendix-lower-Galerkin-bound}, and \eqref{eq:regular-energy-equivalence} imply
\[
 \abs{\ip{A_D^{1/2}f(u_m)}{A_D^{1/2}u_{m,t}}}
 \le C_M\norm{A_Du_m}_2\norm{A_D^{1/2}u_{m,t}}_2
 \le C_M\Escr_{1,m}(t).
\]
Gronwall's inequality yields
\begin{equation}\label{eq:appendix-Galerkin-exp}
 \Escr_{1,m}(t)\le\Escr_{1,m}(0)e^{C_Mt}.
\end{equation}
Integration of \eqref{eq:appendix-nonlinear-high-identity} gives, for every $T>0$,
\begin{equation}\label{eq:appendix-Galerkin-Az}
 \sup_m\left[
 \sup_{0\le t\le T}\Escr_{1,m}(t)
 +D\int_0^T\norm{A_Dz_m(t)}_2^2\,\dd t
 \right]\le C\Escr_1(0)e^{C_MT}.
\end{equation}
In particular, the finite-dimensional solutions cannot blow up.

Fix $T>0$. The estimates give
\[
 u_m,w_m,z_m\quad\text{bounded in }L^\infty(0,T;D(A_D)),
 \qquad
 u_{m,t},w_{m,t}\quad\text{bounded in }L^\infty(0,T;V).
\]
The third equation bounds $z_{m,t}$ in $L^\infty(0,T;V)$. Lemma~\ref{lem:source-estimates} and the equations bound $u_{m,tt},w_{m,tt}$ in $L^\infty(0,T;H)$. Since $D(A_D)\Subset V$ and $V\Subset H$, the Banach-valued Arzel\`a--Ascoli theorem yields, after extraction,
\begin{align}\label{eq:appendix-regular-compactness}
 u_m,w_m,z_m&\to u,w,z&&\text{in }C([0,T];V),\notag\\
 u_{m,t},w_{m,t}&\to u_t,w_t&&\text{in }C([0,T];H).
\end{align}
Local Lipschitz continuity of $f:V\to H$ gives $f(u_m)\to f(u)$ in $C([0,T];H)$, so passage to the limit identifies $(u,w,z)$ with the unique energy solution from Theorem~\ref{thm:local-wp}. The high-order bounds also imply
\[
 u,w,z\in L^\infty(0,T;D(A_D)),
 \qquad
 u_t,w_t\in L^\infty(0,T;V).
\]
Lemma~\ref{lem:regular-composition-general} therefore gives $f(u)\in L^\infty(0,T;V)$. Apply Proposition~\ref{prop:forced-linear-regular} to the limiting system with $G=f(u)$. The resulting regular solution has the same forcing and initial data as the energy solution; uniqueness of the forced linear problem identifies them. Hence $Y\in C([0,T];\Hcal_1)$. The regular composition lemma now gives $f(u)\in C([0,T];V)$, and the equations imply the remaining continuity in \eqref{eq:regularity-class}.

Finally, Proposition~\ref{prop:forced-linear-regular} with $G=f(u)$ gives \eqref{eq:regular-high-identity}. Repeating the estimate above yields \eqref{eq:regular-finite-time-bound}; integration yields \eqref{eq:regular-finite-time-dissipation}. Since $T>0$ was arbitrary, the propagation is global.
\end{proof}

\section*{Acknowledgments}
This research was supported by Basic Science Research Program through the National Research Foundation of Korea (NRF) funded by the Ministry of Education (RS-2022-NR075641).

\end{document}